\documentclass[11pt, a4paper]{amsart}
\usepackage[english]{babel}
\usepackage{amssymb,amsmath,amsthm,mathrsfs,bm}
\usepackage{color}
\usepackage[margin=1in]{geometry}
\usepackage{stmaryrd}

\newtheorem{theorem}[subsection]{Theorem}

\newtheorem{definition}[subsection]{Definition}
\newtheorem{lemma}[subsection]{Lemma}
\newtheorem{remark}[subsection]{Remark}
\newtheorem{proposition}[subsection]{Proposition}
\newtheorem{corollary}[subsection]{Corollary}

\newtheorem*{claim*}{Claim}
\newtheorem*{theorem*}{Theorem}
\def\bal{\begin{aligned}}
\def\eal{\end{aligned}}
\def\be{\begin{equation}}
\def\ee{\end{equation}}
\def\bcs{\begin{cases}}
\def\ecs{\end{cases}}
\def\={\;=\;}
\def\+{\,+\,}
\def\-{\,-\,}

\def\Z{{\mathbb Z}}

\def\Q{{\mathbb Q}}
\def\R{{\mathbb R}}

\def\P{{\mathbb P}}

\def\lb{\llbracket}
\def\rb{\rrbracket}

\def\ord{\mathrm{ord}}
\def\sO{\mathcal{O}} 
\def\sG{\mathcal{G}}

\def\sA{\mathcal{A}}

\def\cartier{\mathscr{C}_p}
\def\fil{\mathscr{F}}

\def\v#1{{\bf #1}}

\def\is{\equiv}
\def\mod#1{({\rm mod}\ #1)}
\def\hat{\widehat}

\definecolor{lightgrey}{rgb}{0.8, 0.84, 0.8}
\title{On $p$-integrality of instanton numbers}
\author{Frits Beukers, Masha Vlasenko}

\address{Utrecht University }
\email{f.beukers@uu.nl}

\address{Institute of Mathematics of the Polish Academy of Sciences}
\email{m.vlasenko@impan.pl}

\thanks{
Work of Frits Beukers was supported by the Netherlands Organisation for 
Scientific Research (NWO),
grant TOP1EW.15.313. Work of Masha Vlasenko was supported by the National Science Centre of 
Poland (NCN), grant UMO-2020/39/B/ST1/00940.}

\begin{document}
\maketitle
\medskip

\centerline{Dedicated to Don Zagier's 70th birthday, in friendship and admiration}
\medskip

\section{Introduction}
The motivation for this paper comes from the striking work 
\cite{COGP91} of Candelas, de la Ossa, Green
and Parkes in the study of mirror symmetry of quintic threefolds from
1991. The story has been told many times, so we will give only a very brief description.
For more
details we like to refer to Duco van Straten's excellent \cite{DS18} and the many
references therein. Our short story starts with the differential operator
\[
L=\theta^4- 5 t(5\theta+1)(5\theta+2)(5\theta+3)(5\theta+4),
\]
where $\theta$ denotes $t\frac{d}{dt}$. 
The unique holomorphic solution to $L(y)=0$
is of hypergeometric type and given by
\[
F_0:=\sum_{n\ge0}\frac{(5n)!}{(n!)^5}t^{n}.
\]
 
The equation $L(y)=0$ has a unique basis of solutions of the form
\[
y_0=F_0,\ y_1=F_0\log t+F_1,\ y_2=\frac12 F_0\log^2t+F_1\log t+F_2,
\]
\[
y_3=\frac16 F_0\log^3t+\frac12 F_1\log^2t+F_2\log t+F_3,
\] 
where $F_0$ is given above and $F_1,F_2,F_3\in t\Q\lb t\rb$. Straightforward computation shows
that the coefficients of $F_1$ are certainly not integral. The surprise is that
$q(t):=t\exp(F_1/F_0)$, expanded as power series in $t$, does
have integer coefficients. The function $q(t)$ is called the 
{\it canonical coordinate}. The inverse power series $t(q)$ 
is called the {\it mirror map}. 
Using this inverse series one can rewrite the solutions $y_i$ as functions of $q$.
In particular,
\[
y_1/y_0 = \log q,\quad y_2/y_0=\frac12 \log^2q+V(q)
\]
for some $V(q)\in q\Q\lb q\rb$.
Let $\theta_q=q\frac{d}{dq}$. Then $K(q):=1+\theta_q^2V$
is called the {\it Yukawa coupling} for reasons coming from theoretical physics. 
In \cite{COGP91} the Yukawa coupling is expanded as a so-called Lambert series in the form
\[
K(q)=1+\sum_{n\ge1}\frac{A_nq^n}{1-q^n}.
\]
Candelas et al conjectured that the numbers $a_n = 5 A_n/n^3$ are integers which are
equal to the (virtual) number of degree $n$ rational curves on a generic quintic threefold in $\P^3$. 
This counting relation is now a part of the mirror symmetry theory
from theoretical physics. The numbers $a_n$ are called {\it instanton numbers}.
In this paper we shall concentrate on their integrality.

It turns out that there are many other examples of differential equations, similar to
the example above, which display integrality properties of the associated instanton numbers.
Motivated by this, Almkvist, van Enckevort, van Straten and Zudilin \cite{AESZ10}
compiled an extensive list of so-called Calabi-Yau operators. For each entry in the list
the authors give experimental evidence that the instanton numbers are $p$-adically integral
for almost all primes $p$, with some obvious exceptions. The quintic example that we started
with belongs to this list. Presumably, each equation in the list is a so-called Picard-Fuchs
equation corresponding to a one parameter family of Calabi-Yau threefolds.

Around 1995 Giventhal \cite{Giv95} showed that instanton numbers can be related to
the geometrical Gromov-Witten invariants of Calabi-Yau threefolds. However, the
integrality of these invariants is not a priori clear. Around 1998 Gopakumar and Vafa
introduced the so-called BPS-numbers for Calabi-Yau threefolds, which include the instanton numbers
as $g=0$ case, see \cite[Def 1.1]{BP01}.
In 2018 Ionel and Parker \cite{IP18} proved that these BPS-numbers
are integers by using methods from symplectic topology. So this opens up a path towards
proving integrality of instanton numbers.

Another approach to the integrality of instanton numbers was the use of Dwork's
ideas in $p$-adic cohomology. This was started by Jan Stienstra \cite{Sti03} in 2003 with 
partial success and later Kontsevich, Schwarz and Vologodsky \cite{KSV06},
\cite{Vol07} laid out the ideas for a completely general approach. However, since \cite{Vol07} is not
published yet, it is not clear to us what the status of the final results is.

One of the main goals of the present paper is to show by
elementary means that for the quintic example and a few
related examples the $a_n$ are $p$-adic integers for almost all $p$. The method
showed up as a by-product of our prior work in the paper Dwork crystals III,
\cite{BeVlIII}, which is a continuation of the preceding papers Dwork crystals I,II
(\cite{BeVl20I},\cite{BeVl20II}). 

We begin by sketching an outline of our approach. To study $p$-integrality 
we let $p$ be an odd prime. 
Let $\sigma:\Z_p\lb t\rb\to\Z_p\lb t\rb$ be a $p$-th power Frobenius lift, which
is a ring endomorphism such that $h(t)^\sigma:=\sigma(h(t))\is h(t^p)\mod{p}$
for all $h\in\Z_p\lb t\rb$.
In this paper we shall take it to be of the form $t^\sigma=t^p(1+ptu(t))$ with
$u(t)\in\Z_p\lb t\rb$. 
We start with a linear differential equation $L(y)=0$ where
\[
L=\theta^n+a_{n-1}(t)\theta^{r-1}+\cdots+a_1(t)\theta+a_0(t) \in\Z_p\lb t\rb[\theta],\quad
\theta=t\frac{d}{dt}.
\]
We also assume that it is of MUM-type (maximal unipotent monodromy) at the origin,
i.e $a_i(0)=0$ for $i=0,\ldots,n-1$. The main examples of such equations are the Picard-Fuchs
equations which are associated to certain one parameter families of algebraic varieties. 
We shall give a more precise definition later on.
For a MUM-type equation there is a unique basis of local solutions 
$y_0,y_1,\ldots,y_{n-1}$ of the form
\be\label{standardbasis}
y_i=F_0\frac{\log^i t}{i!}+F_1\frac{\log^{i-1}t}{(i-1)!}+\cdots+F_{i-1}\log t+F_i,
\ i=0,\ldots,n-1,
\ee
and $F_i\in\Q_p\lb t\rb$, $F_0(0)=1$ and $F_i(0)=0$ if $i>0$. 
We extend the Frobenius lift $\sigma$ to the solutions $y_i^\sigma$ in the obvious
way, i.e. by taking $F_i^\sigma=F_i(t^\sigma)$ and 
\[
\log t^\sigma=p\log t -\sum_{m\ge1}\frac{1}{m}(-p \,t \, u(t))^m.
\]

A familiar property of Picard-Fuchs equations is the presence of a Frobenius
structure for almost all $p$. We give a closely related concrete definition in terms of
its solutions.

\begin{definition}[Frobenius structure]\label{frobeniusstructure}
The equation $L(y)=0$ is said to have a
$p$-adic Frobenius structure if there exists an operator 
\[
\sA=\sum_{i=0}^{n-1}A_i(t)\theta^i\in\Z_p\lb t\rb[\theta]
\]
with $A_0(0)=1$ such that for every solution $y(t)$ of $L(y)=0$ the
composition $\sA(y(t^\sigma))$ is again a solution of this differential equation.

In this paper we will call the numbers $\alpha_i=A_i(0)$ the Frobenius constants.
\end{definition}

\begin{remark}\label{frobenius-in-standard-basis} In the standard basis of solutions~\eqref{standardbasis} the operator $\sA$ acts by the formula
\[
\sA(y_i(t^\sigma)) = p^i \sum_{j=0}^i \alpha_j \, y_{i-j}(t), \quad i=0,\ldots, n-1.
\]
This will be shown in Section \ref{sec:thmproofs}.
\end{remark}

\begin{remark}
Our definition of Frobenius structure is adapted to the integrality questions considered in this paper.
First of all, a
Frobenius structure is usually defined over a ring $R$, which is a (sub)ring of so-called 
analytic elements. These are $p$-adic limits of rational functions. Granted that we take a Frobenius
structure over $\Z_p\lb t\rb$, this would give us $p^iA_i(t)\in\Z_p\lb t\rb$ for $i=0,\ldots,n-1$. This is weaker
than $A_i(t)\in\Z_p\lb t\rb$, which is in our definition. The latter condition is crucial in our paper.
We believe that this condition holds in families of Calabi-Yau varieties which possess sufficiently many symmetries.
\end{remark}

Recall the canonical coordinate $q(t)=t\exp(F_1/F_0)$ and the inverse
series $t(q)$ which we called the mirror map.
It is expected that for Picard-Fuchs equations of
MUM-type these series are in $\Z_p\lb t\rb$ for all but finitely
many primes $p$. So far this has been shown for a
large number of special cases, e.g \cite{LY95}, \cite{KR10}, \cite{Ro14}, \cite{ERR16}. In 
Section \ref{sec:thmproofs} we prove the following.

\begin{theorem}[$p$-integrality of the mirror map]\label{mirror-integrality}
Suppose that $L(y)=0$ has a $p$-adic Frobenius structure. 
Then $\exp(F_1/F_0)\in\Z_p\lb t\rb$. 
\end{theorem}

Let us now suppose that $\exp(F_1/F_0)\in\Z_p\lb t\rb$ and $n\ge3$.
Take $q=t\exp(F_1/F_0)\in t+t^2\Z_p\lb t\rb$.
We choose the special Frobenius lift given by $q^\sigma=q^p$, called the \emph{excellent lift}
in \cite{BeVlIII}. 
Define $\theta_q=q\frac{d}{dq}$. Then the operator $\theta_q^2$ annihilates the functions
$y_1/y_0=\log q$ and $y_0/y_0=1$. 

\begin{definition}
The function $K(q)=\theta_q^2(y_2/y_0)$ is called the Yukawa coupling associated to
the operator $L$. It is equal to $1+\theta_q^2V$, where  we consider 
$V=F_2/F_0 - \frac12(F_1/F_0)^2$ as function of $q$. 
\end{definition}

As in the quintic example we expand
\[
K(q)=1+\sum_{r\ge1}\frac{A_rq^r}{1-q^r}.
\]
In the following theorem we require that the Frobenius constant $\alpha_1$ is $0$.
In general the Frobenius constants $\alpha_i$ are notoriously difficult to compute, but 
fortunately we need only $\alpha_1$.  This is dealt with in the proof of Corollary
\ref{quintic-diagonal}, which in its turn uses \cite[Proposition 7.7]{BeVlIII}.

\begin{theorem}\label{yukawa2}
Suppose that $L(y)=0$ has a $p$-adic Frobenius structure with $\alpha_1=0$.
Then $A_r/r^2\in\Z_p$ for all $r\ge1$.
\end{theorem}

The $p$-integrality of $A_r/r^3$ follows if the differential equation
has order $4$ and is  self-dual. This is precisely the type of operators that is studied
in \cite{AESZ10} and which are named Calabi-Yau operators.
There we find that an operator $L=\theta^4+a_3\theta^3+a_2\theta^2+a_1\theta
+a_0$ is called self-dual if
\[
a_1=\frac12 a_2a_3-\frac18 a_3^3+\theta(a_2)-\frac34 a_3\theta(a_3)-\frac12 \theta^2(a_3).
\]
There are more intrinsic characterizations of self-duality but what matters for us is that
in the case of self-duality we have
\[
\left|\begin{matrix}y_0 & y_3\\ \theta y_0 & \theta y_3\end{matrix}\right|=
\left|\begin{matrix}y_1 & y_2\\ \theta y_1 & \theta y_2\end{matrix}\right|.
\]

\begin{theorem}[$p$-integrality of instanton numbers]\label{yukawa3}
Suppose that the operator $L$ of order $n=4$ is self-dual and the differential equation $L(y)=0$ has a 
$p$-adic Frobenius structure
with $\alpha_1=0$. Then $A_r/r^3\in\Z_p$ for all $n\ge1$.
\end{theorem}

We remark that the theorem is sensitive to the normalization of $t$.
In other words, if we replace $t$ by $2t$ for example,
the theorem is not true anymore. The value $\alpha_1$ has become non-zero. It is worth 
mentioning that the instanton numbers and the Frobenius structure behave nicely
under the change of variables $t \to t^w$. One can easily check that the new Yukawa coupling is
given by $K(q^w)$ and the new constants $\alpha_i$ of the Frobenius structure are given by 
$\alpha^i/w^i$, $i=0,1,\ldots$

Theorems \ref{mirror-integrality}, \ref{yukawa2} and \ref{yukawa3} are proven in Section 
\ref{firstproofs}.

Let us now sketch the set up which we use to construct the Frobenius structure for differential operators. We start with a Laurent polynomial $f(\v x)$ in $\v x=(x_1,\ldots,x_n)$ with coefficients in $\Z[t]$. 
Let $\Delta$ be the Newton polytope of $f$ and $\Delta^\circ$ its interior.
Let $R$ be a ring containing $\Z[t]$.
By $\Omega_f$ we denote the $R$-module generated by the functions
$(k-1)!A/f^k$, where $A$ is a Laurent polynomial with coefficients in $R$
and support in $k\Delta$ and $k\ge1$. 
We define $\Omega_f^\circ$ similarly, consisting of rational 
functions as above, but with support of $A$ in $k\Delta^\circ$.
In the notation of Dwork crystals I,II we use $\Omega_f:=\Omega_f(\Delta)$ 
and $\Omega_f^\circ:=\Omega_f(\Delta^\circ)$. 
By $d\Omega_f$ we define the module spanned by the partial derivatives of the
form $x_i\frac{\partial}{\partial x_i}(\omega)$ with $\omega\in\Omega_f$ and $i=1,\ldots,n$.
We will denote the partial derivations $x_i\frac{\partial}{\partial x_i}$ by $\theta_i$. 

When $f$ has suitable regularity properties, there exists a polynomial
$D_f(t)\in\Z[t]$ such that the quotient $\Omega_f^\circ/d\Omega_f$
is a free module over $R=\Z[t,1/D_f]$ of finite rank. 
It can roughly be identified with the $(n-1)$-st De Rham cohomology of
the zero set of $f$ (although we shall not use this). 

The derivation $\theta=t\frac{d}{dt}$ on $\Z[t]$ can be extended to 
$\Omega_f$ in the obvious way. On easily checks that $\theta$ sends $\Omega_f^\circ$
to itself as well as $d\Omega_f$.
Hence $\theta$ maps $\Omega_f^\circ/d\Omega_f$ to itself.
This gives us the so-called Gauss-Manin connection on $\Omega_f^\circ/d\Omega_f$.

Our basic example is
\[
f=1-t\left(x_1+x_2+x_3+x_4+\frac{1}{x_1x_2x_3x_4}\right).
\]
Then $\Omega_f^\circ /d \Omega_f$ is a free $\Z[t,1/((5t)^6-5t)]$-module of
rank 4 with a basis given by $\theta^i(1/f)$ with $i=0,1,2,3$.
Clearly $\theta^4(1/f)$ depends on these and the
relation is given by $L(1/f)\in d\Omega_f$, where $L$ is the linear differential
operator
\[
L=\theta^4-(5t)^5(\theta+1)(\theta+2)(\theta+3)(\theta+4).
\]
This operator is related to
the hypergeometric operator with which we began our introduction via the change of variable 
$t \to t^5$. We call it the {\it quintic example}.

Our second main example illustrates the use of symmetries of $f$.
Let  $\sG$ be a finite group of monomial substitutions under which $f$ is invariant.
A monomial substitution has the form $x_i\to \v x^{\v a_i}, i=1,\ldots,n$
where $\v a_i$ are vectors of integers which form a basis of $\Z^n$.
We denote the submodule of invariant rational functions by $(\Omega_f^\circ)^{\sG}$.
Consider
\[
f=1-t\left(x_1+\frac{1}{x_1}+x_2+\frac{1}{x_2}+x_3+\frac{1}{x_3}+x_4+\frac{1}{x_4}\right).
\]
It turns out that $\Omega_f^\circ /d \Omega_f$ has rank $10$ as module over
$\Z[t,(2t(1-80t^2+1024t^4))^{-1}]$. However, $f$ is invariant
under the group $\sG$ of monomial transformations generated by the permutations of $x_1,x_2,x_3,x_4$
and $x_i\to x_i^{\pm1}$ for all $i$. We denote the $\sG$-invariant elements of $\Omega_f^\circ$ 
by $(\Omega_f^\circ)^\sG$. The quotient module 
$(\Omega_f^\circ)^\sG /d \Omega_f$ turns out to have rank $4$ and is generated by 
$\theta^i(1/f)$ for $i=0,1,2,3$.
Then $L(1/f)\in d\Omega_f$ with
\begin{eqnarray*}
L &=&(1024t^4-80t^2+1)\theta^4+64(128t^4-5t^2)\theta^3\\
&&+16(1472t^4-33t^2)\theta^2+32(896t^4-13t^2)\theta+128(96t^4-t^2).
\end{eqnarray*}
This is \#16 in the database of Calabi-Yau equations \cite{AESZ10} with $z=t^2$.
We call it the {\it diagonal example}.\label{setup}

Prompted by these examples we now restrict the Laurent polynomial $f(\v x)$ to
a polynomial of the form $f=1-tg(\v x)$, where $g$ is a Laurent polynomial
in the variables $\v x=(x_1,\ldots,x_n)$ with coefficients in $\Z$.
We also assume that its Newton polytope $\Delta \subset \R^n$ is \emph{reflexive}.

We remind the reader that a lattice polytope $\Delta\subset\R^n$
is reflexive if it is of maximal dimension, contains $\v 0$ and each of its
codimension~1 faces can be given by an equation $\sum_{i=1}^na_ix_i=1$ 
with coefficients $a_i\in\Z$. It follows from this definition that
\emph{$\v 0$ is the unique lattice point in $\Delta^\circ$}. Reflexivity is a 
requirement needed for the $p$-adic considerations later on.

Let us remark that in Dwork crystals II \cite{BeVl20II} we prove congruences of 
Dwork type for the power series $y_0(t)=\sum_{m\ge0}g_mt^m$, with 
$g_m=\text{constant term of }g(\v x)^m$,
which are associated to $\Omega_f$ for these particular $f$. The above relation 
$L(1/f) \in d \Omega_f$ implies that $y_0(t)$ is a power series solution to the
differential equation $L(y)=0$. This follows from the properties of the period map
constructed in \cite[\S 2]{BeVl20II}. See also the explicit families in 
\cite[\S 7]{BeVlIII}.

Another property of reflexive polytopes which we will use is that to every
lattice point $\v u \in \Z^n$ there is a unique integer $d \ge 0$ such that $\v u$
lies on the boundary of $d \Delta$. We call this integer $d$ the degree of both $\v u$
and its respective monomial $\v x^\v u$. For a Laurent polynomial 
$A \in R[x_1^{\pm 1},\ldots,x_n^{\pm 1}]$ its degree $\deg(A)$ is the maximum degree
of the monomials it contains. Suppose that $R$ has the form $\Z[t,1/D_f(t)]$ as above with $D_f(t)\in\Z[t]$
and $D_f(0)\ne0$. A Laurent polynomial 
$A(x) = \sum_{\v u}a_{\v u}\v x^{\v u}$ with coefficients in 
$R$ is called {\it admissible} if for every monomial $\v x^{\v u}$ 
in $A$ the $t$-adic order of its coefficient $a_{\v u}$ is greater or
equal to $\deg(\v x^{\v u})$. For example, $f(\v x)=1-tg(\v x)$ is admissible.

Throughout the paper we will work with the submodule 
\[
\sO_f  = \left\{ (k-1)!\frac{A(\v x)}{f(\v x)^k} \;\Big| \; k \ge 1, \, A 
\text{ is admissible}, \, Supp(A) \subset k \Delta  \right\} \subset \Omega_{f}.
\]
We will call elements of $\sO_f$ admissible rational functions. 
The submodule $\sO_f^\circ \subset \sO_f$ is defined by the stronger condition 
$Supp(A) \subset k \Delta^\circ$. The submodule of derivatives 
$d \sO_f \subset \sO_f$ is the $R$-module generated by 
$x_i\frac{\partial}{\partial x_i}(\omega)$ for $\omega \in \sO_f$ and $i=1,\ldots,n$.

The main motivation for choosing $\sO_f$ is that in our examples the quotient module
$\sO_f^\circ/d\sO_f$ is of finite rank over the ring $R=\Z[t,1/D_f(t)]$ for some $D_f(t)\in\Z[t]$
with $D_f(0)\ne0$. So when $p$ does not divide $D_f(0)$, this ring can be embedded in $\Z_p\lb t\rb$,
which is convenient for our $p$-adic considerations.

Letting $\sG$ be a group of monomial automorphisms of $f(\v x)$, 
we denote the submodule of $\sG$-invariant admissible rational functions 
by $(\sO_f^\circ)^{\sG}$. 
It is not hard to verify that $\gamma \circ\theta=\theta\circ \gamma$ for all $\gamma\in \sG$.
In particular this implies that $\theta$ maps $(\sO_f^\circ)^\sG/d\sO_f$ to itself.
We now assume that
$M:=(\sO_f^\circ)^\sG/d\sO_f$ is free of finite rank $n$ and is generated by the
elements $\theta^i(1/f)$ for $i=0,1,\ldots,n-1$. We actually need more refined 
assumptions to be satisfied, in which case we call $M$ a 
\emph{cyclic $\theta$-module of MUM-type}. See Definition \ref{cyclicMUM} for the precise formulation.
Just as in the above two examples we can associate to a cyclic $\theta$-module of
MUM-type an $n$-th order linear differential operator $L$ of MUM-type such that $L(1/f) \in d\sO_f$. We
call $L(y)=0$ the \emph{Picard-Fuchs differential equation} associated to $M$. These equations
are precisely the type of differential equations that we want to consider.

To introduce the Frobenius structure on the Picard-Fuchs equation we choose
an odd prime $p$ and assume that $\Z[t,1/D_f(t)]$ can be embedded in
$\Z_p\lb t\rb$. Let $\sigma$ be a Frobenius lift of $\Z_p\lb t \rb$ satisfying the assumption
$t^\sigma \in t^p(1 + p t \Z_p\lb t \rb)$ which was introduced earlier. We recall the Cartier operator
$\cartier$ which was introduced
in Dwork crystals I, \cite{BeVl20I}. Using a vertex
$\v b$ of $\Delta$ we can expand any element of $\sO_f$ as a formal Laurent
series with support in the cone spanned by $\Delta-\v b$. On such Laurent series
the Cartier operator acts as
\be\label{Cartier-on-expansions}
\cartier:\sum_{\v k\in C(\Delta-\v b)}a_{\v k}\v x^{\v k}\mapsto
\sum_{\v k\in C(\Delta-\v b)}a_{p\v k}\v x^{\v k}.
\ee
We extend $\sigma$ to $\sO_f$ by applying it to the coefficients of rational functions.
For example, $f^\sigma(\v x)=1-t^\sigma g(\v x)$.
Let $\sO_{f^\sigma}$ be the $\Z_p\lb t \rb$-module generated by the
$\sigma$-images of the elements of $\sO_f$.
The numerators of these images are admissible with respect to $t^\sigma$ instead of $t$.
Then, in \cite[Proposition 2.7]{BeVlIII}
it is shown that $\cartier$ maps $\hat\sO_f^\circ$ to $\hat\sO_{f^\sigma}^\circ$,
where $\hat\sO_f^\circ$ denotes the $p$-adic completion of $\sO_f^\circ$ and similarly for 
$\hat\sO_{f^\sigma}^\circ$. Moreover, when
\[
p \nmid \#\sG
\]
our Cartier operator maps $\sG$-symmetric functions $(\hat \sO_f^\circ)^\sG$ into $(\hat \sO_{f^\sigma}^\circ)^\sG$. This can be seen by observing that the symmetrization
operator $\omega \mapsto \frac{1}{\#\sG}\sum_{\gamma\in\sG}\gamma(\omega)$ commutes with $\cartier$. 

One can easily deduce from~\eqref{Cartier-on-expansions} that $\cartier \circ \theta_i = p \, \theta_i \circ \cartier$ for $i=1,\ldots,n$, from which it follows that our Cartier operation maps derivatives $d\hat\sO_f$ into $d\hat\sO_{f^\sigma}$. Hence $\cartier$ descends to a $\Z_p\lb t \rb$-linear map 
\[
\cartier: (\hat \sO_f^\circ)^\sG /d\hat\sO_f \to (\hat \sO_{f^\sigma}^\circ)^\sG / d\hat\sO_{f^\sigma}.
\]
When $(\sO_f^\circ)^\sG/d\sO_f$ are cyclic $\theta$-modules of rank $n$, this implies existence of power series $\lambda_0,\ldots,\lambda_{n-1} \in \Z_p\lb t \rb$ such that
\[
\cartier(1/f) = \sum_{i=0}^{n-1} \lambda_i(t) (\theta^i(1/f))^\sigma \mod {d\hat\sO_{f^\sigma}}.
\]
We shall explicitly construct a $p$-adic Frobenius structure $\sA$ for the Picard-Fuchs differential equation $L(y)=0$ using the series $\lambda_i(t)$ in this formula. It is crucial for our construction to know that the derivatives $\theta^i(1/f)$, $i=0,\ldots,n-1$ are linearly independent not just modulo $d\sO_f$ but also modulo $d\hat\sO_f$. This would particularly imply that the coefficients $\lambda_i(t)$ in the above formula are defined uniquely. Unfortunately, we were only able to show independence of derivatives modulo $d\hat\sO_{f}$ under the so-called $n$-th Hasse--Witt condition, which is recalled in Section \ref{sec:HWC}. 

\begin{proposition}\label{cartier2frobenius}
Suppose $f=1-tg(\v x)$ has a reflexive Newton polytope $\Delta$ and let $\sG$ be a group of monomial
automorphisms of $g$. Suppose that $(\sO_f^\circ)^\sG /d \sO_f$ is a cyclic $\theta$-module of MUM-type of
rank $n$.
Suppose $p\nmid\#\sG\times D_f(0)$ and $p>n$. Then, if the $n$-th Hasse--Witt condition holds,
the corresponding Picard-Fuchs equation $L(y)=0$ has a $p$-adic Frobenius structure with $\alpha_i=p^{-i}\lambda_i(0)$, $0 \le i \le n$.
\end{proposition}

We give a proof of this proposition in Section \ref{firstproofs}.
In the literature one sees that existence of a $p$-adic Frobenius structure for Picard-Fuchs equations
is a well-known phenomenon. But as we said before, we need something a bit stronger in the form of
Definition \ref{frobeniusstructure}. We believe that families of Calabi-Yau
varieties with sufficient symmetry have a Frobenius structure that satisfies this definition,
but we are only able to give a few explicit examples.
Below we give two examples to which Proposition~\ref{cartier2frobenius} applies.
\smallskip

{\bf Simplicial family.}
\smallskip

Let
\[
g(\v x)=x_1+\cdots+x_n+\frac{1}{x_1\cdots x_n}.
\]
We call the corresponding family of varieties $1-tg(\v x)=0$ the simplicial family, refering to the
shape of the Newton polytope. For $\sG$ we take the group generated by the permutations of
$x_1,\ldots,x_n$ and $x_1\to (x_1\cdots x_n)^{-1}$. It is isomorphic to $S_{n+1}$.
In Section \ref{sec:simplicial} we show this family gives rise to a cyclic $\theta$-module
of rank $n$ with $D_f(t)=(n+1)(1-((n+1)t)^{n+1})$.
The corresponding Picard-Fuchs operator reads
\[
\theta^n - ((n+1)t)^{n+1} (\theta+1)\ldots(\theta+n),
\]
which is of hypergeometric origin and of MUM-type.
This family is related to the so-called Dwork families $X_0^{n+1}+\cdots+X_n^{n+1}=\frac{1}{t}
X_0\cdots X_n$ which one often finds in the literature. Simply replace $x_i$ in the simplicial
family by $X_i^{n+1}/(X_0\cdots X_n)$ for every $i$. The case $n=4$ is of course the quintic family.
\smallskip

{\bf Hyperoctahedral family.  }
\smallskip

Let
\[
g(\v x)=x_1+\frac{1}{x_1}+x_2+\frac{1}{x_2}+\cdots+x_n+\frac{1}{x_n}.
\] 
The corresponding family of varieties $1-tg(\v x)=0$ is called the hyperoctahedral family,
also after the shape of its Newton polytope. For the group $\sG$ we take the group
generated by the permutations of $x_1,\ldots,x_n$ and $x_i\to x_i^{\pm1}$ for $i=1,\ldots,n$. In Section
\ref{sec:hyperoctahedral} we will show that there
exists a polynomial $D_f(t)$ such that over the ring $\Z[t,1/D_f(t)]$ the module $(\sO_f^\circ)^{\sG}/d\sO_f$ is
generated by $\theta^i(1/f)$ for $i=0,\ldots,n-1$. There we also construct a Picard-Fuchs operator
$L$, which turns out to be of MUM-type. If $L$ were irreducible in $\Q(t)[\theta]$ then 
$(\sO_f^\circ)^{\sG}/d\sO_f$ would be a rank~$n$ cyclic $\theta$-module of MUM-type of rank $n$.
When $n=4$ we recover the diagonal family introduced earlier. 

In Section~\ref{sec:HWC} we will show that the $n$-th Hasse--Witt condition holds for the simplicial and
hyperoctahedral families.
As a consequence of all this we find the following.

\begin{corollary}\label{quintic-diagonal}
In the simplicial family the numbers $A_r/r^2$ are $p$-adically
integral for all $p>n+1$. 

In the hyperoctahedral case we have the same result for all $p>n$ but under the condition that the
Picard-Fuchs operator is irreducible in $\Q(t)[\theta]$.

In the 'quintic example' and the 'diagonal example' above
the instanton numbers are $p$-integral for every prime $p>5$ in the quintic case
and $p\ge5$ in the diagonal case.
\end{corollary}

\begin{remark}
We conjecture that in the hyperoctahedral case the Picard-Fuchs operator is irreducible for all $n\ge2$.
We are indebted to Jaques-Arthur Weil who verified the
irreducibility when $2\le n\le10$ with the aid of the DEtools package in Maple.
\end{remark}

\begin{proof}[Proof of Corollary \ref{quintic-diagonal}]
In Sections \ref{sec:simplicial} and \ref{sec:hyperoctahedral} we find that both examples produce
cyclic $\theta$-modules of
MUM-type and rank $n$. A quick check shows that the Newton polytopes satisfy the conditions of 
Theorem~\ref{HWC}, which then tells us that the $n$-th Hasse--Witt condition holds.
Hence, by Propostion \ref{cartier2frobenius} both examples have Picard-Fuchs equations with a
Frobenius structure. In Section~\ref{sec:vanishing} we show that $\alpha_1=0$.
Namely, by Lemma~\ref{vanishing-lemma} we have $\alpha_1=p^{-1}
\lambda_1(0)=p^{-1}\widetilde\lambda_1(0)$, where $\widetilde\lambda_1(0)$ is the quantity computed in
\cite[Proposition 7.7]{BeVlIII}. If $\sG$ acts transitively on the vertices of $\Delta$,
$t^\sigma/t^p \in 1 + p t\Z_p\lb t \rb$ and the vertex coefficient of $g(\v x)$ is $1$,
this proposition gives $\widetilde\lambda_1(0)=0$.
The Corollary now follows from Theorems \ref{yukawa2} and \ref{yukawa3}.
\end{proof}

By the time the present paper was finished, we realized that the numbers $A_r/r^2$ in
the case $n=5$ of the simplicial family are the so-called $g=0$ BPS-numbers 
corresponding to this family of Calabi-Yau fourfolds. See \cite[\S6.1]{KP08}.
This partly answers a question posed to us by Martijn Kool about BPS-numbers for Calabi-Yau
fourfolds in one instance.
\medskip

{\bf Acknowledgement}: We like to thank Martijn Kool and Duco van Straten for their
explanations about BPS-numbers. We also thank Jacques-Arthur Weil for the verification
of the irreducibility of the Picard-Fuchs operator in the hyperoctahedral case for small $n$. 

\section{Proof of Remark \ref{frobenius-in-standard-basis} and
Theorems \ref{mirror-integrality}, \ref{yukawa2}, \ref{yukawa3}}
\label{sec:thmproofs}
We start with a classical lemma.

\begin{lemma}[Dieudonn\'e-Dwork lemma]\label{DieudonneDwork}
Let $g\in t\Q_p\lb t\rb$. Then $g(t)-\frac{1}{p}g(t^\sigma)\in\Z_p\lb t\rb$ if and only if
$\exp(g)\in\Z_p\lb t\rb$.
\end{lemma}

For a proof see \cite[Lemma 7.10]{BeVlIII}. 
The following criterion allows us to decide $p$-integrality of the numbers $A_r/r^2$ or $A_r/r^3$.

\begin{lemma}\label{instanton-division}
Let $G(q)\in1+q\Q_p\lb q\rb$. Consider its Lambert series expansion
\[
G(q)=1+\sum_{n\ge1}\frac{h_nq^n}{1-q^n}.
\]
Suppose there exists $s \ge 1$ and $\phi\in\Z_p\lb q\rb$
such that $G(q^p)-G(q)=\theta_{q}^s\phi.$
Then $h_n/n^s \in \Z_p$ for all $n \ge 1$. 

\end{lemma}
\begin{proof}
Write $G(q)=1+\sum_{n\ge1}g_nq^n$. Then $G(q^p)-G(q)=\theta_q^s\phi$ implies
that $g_n\is g_{n/p}\mod{n^s\Z_p}$ for all $n\ge1$ (use the convention $g_{n/p}=0$
if $p$ does not divide $n$). One has $g_n=\sum_{d|n}h_d$ and its M\"obius inversion reads 
$h_n=\sum_{d|n}\mu(d)g_{n/d}$,
where $\mu(n)$ is the M\"obius function. Write
\begin{eqnarray*}
h_{n}&=&\sum_{d|n}\mu(d)g_{n/d}\\
&=&\sum_{d|(n/p),d\not\equiv0(p)}\mu(dp)g_{n/pd}+\mu(d)g_{n/d}\\
&=&\sum_{d|(n/p),d\not\equiv0(p)}\mu(d)(g_{n/d}-g_{n/pd})
\end{eqnarray*}
The latter sum is in $n^s \Z_p$ because of the congruences for $g_n$.
\end{proof}

The following lemma gives a criterion to recognize elements of $\Z_p\lb t\rb$.

\begin{lemma}\label{dwork-expansion}
Let $u(t)\in 1+t\Q_p\lb t\rb$ and suppose there exists a linear operator
$\mathcal{A}\in\Z_p\lb t\rb[\theta]$ such that $\mathcal{A}(u^\sigma)=u$.
Then $u\in1+t\Z_p\lb t\rb$.
\end{lemma}

\begin{proof}
Write $\mathcal{A}=\sum_{i \ge 0} A_i(t) \theta^i$.
Then, by setting $t=0$ in $\mathcal{A}(u^\sigma)=u$ we see that $A_0(0)=1$.
Consider the recursion $u_{i+1}=\mathcal{A}(u_i^\sigma),i\ge0$ and $u_0=1$.
Since $A_0(0)=1$ we get that $u_i\in1+t\Z_p\lb t\rb$ for all $i\ge0$.

By induction on $i$ we now prove that $u_i\is u\mod{t^{p^i}}$ for all $i\ge0$. For
$i=0$ this is trivial. For $i\ge0$ we see that $u_i\is u\mod{t^{p^i}}$ implies
$u_i^\sigma\is u^\sigma\mod{t^{p^{i+1}}}$. Application of $\mathcal{A}$ on both
sides gives $u_{i+1}\is u\mod{t^{p^{i+1}}}$. We conclude that $u\in1+t\Z_p\lb t\rb$.
\end{proof}

\begin{proof}[Proof of Remark \ref{frobenius-in-standard-basis}]
Choose $i$. Then $\mathcal{A}(y_i^\sigma)$ is a solution
of $L(y)=0$. So there exist constants $b_0,\ldots,b_{n-1}$ such that
\be\label{frobeniustransform}
\sA(y_i^\sigma)=b_0y_0+b_1y_1+\cdots+b_{n-1}y_{n-1}.
\ee
This is an equality in the ring $\Q_p\lb t\rb[\log t]$. When we say that we take
the constant term of an element of $\Q_p\lb t\rb[\log t]$ we simply mean that we evaluate
the $\Q_p\lb t\rb$-coefficients at $t=0$
but keep de powers of $\log t$. This gives us a ring homomorphism from $\Q_p\lb t\rb[\log t]$
to $\Q_p[\log t]$. 
In particular the constant term of $y_i$ is $\frac{1}{i!}\log^it$ and the
constant term of $y_i^\sigma$ is $\frac{p^i}{i!}\log^it$.
It is not hard to see that the operation of taking the constant terms commutes with $\theta$. 
Hence if we take constant terms on both sides of \eqref{frobeniustransform}
we get
\[
\sum_{j=0}^{n-1}A_j(0)\theta^j\left(p^i\frac{\log^it}{i!}\right)
=\sum_{l=0}^{n-1}b_l\frac{\log^lt}{l!}.
\]
Elaboration of the left hand side gives us
\[
p^i\sum_{j=0}^iA_j(0)\frac{\log^{i-j}t}{(i-j)!}=\sum_{l=0}^{r-1}b_l\frac{\log^lt}{l!}.
\]
We conclude that $b_l=p^iA_{i-l}(0)$ if $l\le i$ and $b_l=0$ if $l>i$. We put
$\alpha_i=A_i(0)$.
\end{proof}

\begin{proof}[Proof of Theorem \ref{mirror-integrality}]
Due to Remark \ref{frobenius-in-standard-basis} there exists $\mathcal{A}\in\Z_p\lb t\rb[\theta]$ such that
$\mathcal{A}(y_0^\sigma)=y_0$.
By Lemma \ref{dwork-expansion} we have $y_0\in 1+t\Z_p\lb t\rb$.
We also have that $\mathcal{A}(y_1^\sigma)=p(y_1+\alpha_1y_0)$.
Let us define $v=y_1/y_0$ and rewrite $\mathcal{A}(y_1^\sigma)$ as
\[
\mathcal{A}(v^\sigma y_0^\sigma)=v^\sigma\mathcal{A}(y_0^\sigma)+
\mathcal{A}_1(\theta(v^\sigma))
\]
for some operator $\mathcal{A}_1\in\Z_p\lb t\rb[\theta]$. Since $\mathcal{A}(y_0^\sigma)=y_0$
we get
\be\label{frobenius-v}
v^\sigma y_0+\mathcal{A}_1(\theta(v^\sigma))=p(v+\alpha_1) y_0.
\ee
Divide on both sides by $y_0$ and apply $\theta$. We get, after division by $p$,
\[
\mathcal{A}_2(\frac{1}{p}\theta(v^\sigma))=\theta v
\]
for some operator $\mathcal{A}_2\in \Z_p\lb t\rb[\theta]$. Write $\frac{1}{p}\theta(v^\sigma)=
\frac{1}{p}\frac{\theta(t^\sigma)}{t^\sigma}(\theta v)^\sigma$
and note that
$z(t):=\frac{1}{p}\frac{\theta(t^\sigma)}{t^\sigma}\in\Z_p\lb t\rb$.
Hence $\mathcal{A}_2(z(t)(\theta v)^\sigma)=\theta v$. 
Moreover, $(\theta v)(0)=1$. Hence $\theta v\in1+t\Z_p\lb t\rb$ by Lemma \ref{dwork-expansion}.
This also implies $\theta(v^\sigma)=pz(t)(\theta v)^\sigma \in p\Z_p\lb t\rb$.
Using this in \eqref{frobenius-v} and division by $py_0$ yields
\[
\frac{1}{p}v^\sigma-v\in \Z_p\lb t\rb,
\]
hence
\[
\frac{1}{p}\frac{F_1^\sigma}{F_0^\sigma}-\frac{F_1}{F_0}\in\Z_p\lb t\rb.
\]
The Dwork-Dieudonn\'e lemma then implies that $\exp(F_1/F_0)\in\Z_p\lb t\rb$.
\end{proof}

\begin{proof}[Proof of Theorem \ref{yukawa2}]
By Theorem \ref{mirror-integrality} we have $q=\exp(y_1/y_0)\in t+t^2\Z_p\lb t\rb$.
We then have $\Z_p\lb q\rb=\Z_p\lb t\rb$. Consider the special Frobenius lift given
by $q^\sigma=q^p$.
Since $L$ has a Frobenius structure, the same holds for the operator $L\circ y_0$ which
has $1=y_0/y_0,y_1/y_0,\ldots$ as solutions. Simply replace the operator $\mathcal{A}$ in
Definition \ref{frobeniusstructure}
by $\frac{1}{y_0}\circ \mathcal{A}\circ y_0^\sigma$. This change does not affect the
value of $\alpha_1$, which is still $0$. We write these functions in terms of $q$.
We have $y_0/y_0=1,y_1/y_0=\log q,y_2/y_0=\frac12\log^2q+V_2(q)$.
Frobenius structure with $\alpha_1=0$ for this operator implies that there is a
$\mathcal{A}\in\Z_p\lb q\rb[\theta_q]$ such that
\be\label{fr-structure-rank-3}
\bal
&\mathcal{A}\left(1\right)=1, \quad \mathcal{A}\left(\log q\right)= \log q,\\
&\mathcal{A}\left(\frac12\log^2(q^p)+V_2(q^p)\right)=
p^2\left(\frac12\log^2q+V_2(q)+\alpha_2\right).
\eal\ee
It follows from the first two identities that $\mathcal{A}=1+\mathcal{A}_2\theta_q^2$ with some
operator $\mathcal{A}_2 \in \Z_p\lb q\rb[\theta_q]$. The second identity above then turns into
\be\label{frobenius-v2}
V_2(q^p)+\mathcal{A}_2(p^2+p^2(\theta_q^2V_2)(q^p))=p^2(V_2(q)+\alpha_2).
\ee
Apply $\theta_q^2$, divide by $p^2$ and add $1$ on both sides. We obtain
\[
\left( 1+\theta_q^2\mathcal{A}_2\right)(1+(\theta_q^2V_2)(q^p))=1+(\theta_{q}^2V_2)(q).
\]
Using Lemma \ref{dwork-expansion} we find that $K(q)=1+(\theta_q^2V_2)(q)\in \Z_p\lb q\rb$.
Using this in \eqref{frobenius-v2} we obtain
\[
\frac{1}{p^2}V_2(q^p)-V_2(q)\in \Z_p\lb q\rb.
\]
Denote this function by $\phi(q)$. Apply $\theta_q^2$ to get 
$K(q^p)-K(q)=\theta_q^2\phi(q)$.
Lemma \ref{instanton-division} with $G(q)=K(q)$ and $s=2$ then implies our theorem.
\end{proof}

\begin{proof}[Proof of Theorem \ref{yukawa3}]
We again work over the ring $\Z_p\lb q\rb$ and Frobenius lift given by $q^\sigma=q^p$.
Again the operator $L\circ y_0$ has a Frobenius structure with $\alpha_1=0$. As functions of $q$
its first four solutions are given by
\[
y_0/y_0=1,\ y_1/y_0=\log q,\ y_2/y_0=\frac12\log^2q+V_2(q),
\ y_3/y_0=\frac16\log^3q+V_2(q)\log q+V_3(q).
\] 
The self-duality relation
\[
\left|\begin{matrix}y_0 & y_3\\ \theta_q y_0 &\theta_q y_3\end{matrix}\right|=
\left|\begin{matrix}y_1 & y_2\\ \theta_q y_1 &\theta_q y_2\end{matrix}\right|
\]
evaluates to $2V_2+\theta_qV_3=0$. 
Existence of the Frobenius structure with $\alpha_1=0$ the operator $L\circ y_0$ implies
that there exists
$\mathcal{A}\in\Z_p\lb q\rb[\theta_q]$ such that the identities~\eqref{fr-structure-rank-3}
hold along with the additional identity
\[
\mathcal{A}\left(\frac16\log^3(q^p)+V_2(q^p)\log(q^p)+V_3(q^p)\right)=
p^3\left(\frac16\log^3q+V_2(q)\log q+V_3(q)+\alpha_2\log q+\alpha_3\right).
\]
Here $\alpha_2,\alpha_3\in\Z_p$. Following the arguments from the proof of Theorem~\ref{yukawa2} we can
write $\mathcal{A}=1+\mathcal{A}_2 \theta_q^2$ with some differential operator $\mathcal{A}_2$
satisfying~\eqref{frobenius-v2}. Our additional identity now becomes
\[
pV_2(q^p)\log q+V_3(q^p)+\mathcal{A}_2
\left[p^3\log q+2 p^2(\theta V_2)(q^p)+p^3(\theta^2V_2)(q^p)\log q\right.
\]
\[
\left.+p^2(\theta^2V_3)(q^p)\right]=
p^3\left(V_2(q)\log q+V_3(q)+\alpha_2\log q+\alpha_3\right).
\]
Use the relation $2\theta_qV_2+\theta_q^2V_3=0$, which follows from the
duality relation $2V_2+\theta_qV_3=0$. We get after division by $p^3$,
\[
\frac{1}{p^3}V_3(q^p)+\frac{1}{p^2}V_2(q^p)\log q+
\mathcal{A}_2(K(q^p)\log q)=(V_2(q)+\alpha_2)\log q+V_3(q)+\alpha_3.
\]
We know from Theorem \ref{yukawa2} that $K(q)\in\Z_p\lb q\rb$. Terms with $\log q$ cancel due to the
identity~\eqref{frobenius-v2} and the remaining terms yield
\[
\frac{1}{p^3}V_3(q^p)-V_3(q)\in\Z_p\lb q\rb.
\]
We denote this function by $\psi(q)$ and 
apply $\theta_q^3$ to get $(\theta_q^3V_3)(q^p)-(\theta_qV_3)(q)=\theta_q^3\psi$.
Hence $K(q^p)-K(q)=-\frac12\theta_q^3\psi$. 
Lemma \ref{instanton-division} with $G(q)=K(q)$ and $s=3$ then implies our theorem.
\end{proof}

\section{Checking higher Hasse--Witt conditions}\label{sec:HWC} 

We recall our setup which was introduced starting on page~\pageref{setup}.
We have a Laurent polynomial $f(\v x)=1-t g(\v x)$
with $g(\v x)\in \Z[ x_1^{\pm 1},\ldots,x_n^{\pm 1}]$ and such that its Newton polytope 
$\Delta \subset \R^n$ is \emph{reflexive}. We pick a prime $p$ and work over the ring 
$R=\Z_p\lb t \rb$ with a fixed Frobenius lift $t^\sigma \in t^p(1+p t \Z_p\lb t \rb)$.
The $R$-module $\sO_f$ consists of rational functions of the form $(k-1)!A(\v x)/f(\v x)^k$
with $k \ge 1$ and \emph{admissible} polynomials $A(\v x)$ supported in $k\Delta$. 
We have the $R$-linear 
Cartier operator on the $p$-adic completions $\cartier:\hat\sO_f\to \hat\sO_{f^\sigma}$,
where $f^\sigma(\v x) = 1 -t^\sigma g(\v x)$.
In \cite[\S3]{BeVlIII} we introduced for $k\ge1$ the so-called $k$-th formal derivates in $\hat\sO_f$ by
\[
\fil_k:=\{\omega\in\hat\sO_f|\cartier^s(\omega)\in p^{ks}\hat\sO_{f^{\sigma^s}}
\ \text{for all $s\ge1$}\}.
\] 
The main result in \cite{BeVlIII} is Theorem 4.2 together with its Corollary 5.9, which state
that if $k<p$ and the Hasse--Witt determinants $hw^{(\ell)}$ are in $R^\times$ for $\ell=1,\ldots,k$, then
\be\label{O-f-decomposition-k}
\hat\sO_f\cong \sO_f(k)\oplus\fil_k,
\ee
where $\sO_f(k)$ is the $R$-module generated by the rational functions $t^{\deg(\v u)}\v x^{\v u}/f(\v x)^k$
with $\v u\in k\Delta$. Similarly we have $\hat\sO_{f^\sigma}\cong \sO_{f^\sigma}(k)\oplus\fil_k^\sigma$,
where 
\[
\fil_k^\sigma:=\{\omega\in\hat\sO_{f^\sigma}|\cartier^s(\omega)\in p^{ks}\hat\sO_{f^{\sigma^{s+1}}}
\ \text{for all $s\ge1$}\}.
\]

For a set $S \subset \R^n$ we denote by $S_\Z$ the set of integral points in $S$, $S_\Z = S \cap \Z^n$.
The Hasse--Witt determinants $hw^{(k)}$ are defined as follows. 

\begin{definition}Let $1 \le k < p$. The $k$th \emph{Hasse--Witt matrix} $HW^{(k)}$ is the square matrix
indexed by $\v u, \v v \in (k \Delta)_\Z$, the $(\v u, \v v)$-entry being the coefficient of 
$\v x^{p \v v - \v u}$ in the polynomial
\be\label{HW-k-polynomial}
F^{(k)}(\v x) = f(\v x)^{p-k} \sum_{r=0}^{k-1} (f^\sigma(\v x^p) - f(\v x)^p)^r f^\sigma(\v x^p)^{k-1-r}
\ee
divided by $t^{p \deg(\v v)-\deg(\v u)}$.
\end{definition}

We note that the entries of $HW^{(k)}$ are in $\Z_p \lb t \rb$. This is because $F^{(k)}(\v x)$ is
admissible and $\deg(p \v v - \v u) \ge p \deg(\v v) - \deg(\v u)$. We introduced these matrices
in~\cite[\S 5]{BeVlIII}. They arise naturally when one wants to decribe the Cartier action on
admissible rational functions modulo $p^k$. Namely, for any $\v u \in (k \Delta)_\Z$ one has
\[
\cartier \frac{t^{\deg(\v u)}\v x^\v u}{f(\v x)^k} = \sum_{\v v \in (k \Delta)_\Z}
HW^{(k)}_{\v u \v v} \; \frac{t^{p \deg(\v v)}\v x^\v v}{f^\sigma(\v x)^k} \mod {p^k \hat\sO_{f^\sigma}},
\]
see \cite[(11)]{BeVlIII}. It follows from~\cite[Proposition 5.7]{BeVlIII} that $\det(HW^{(k)})$
is divisble by $p^{L(k)}$ where
\[
L(k) = \sum_{\ell=1}^{k} \left( \# (k \Delta)_\Z - \# (\ell \Delta)_\Z \right) = 
\sum_{\v u \in (k \Delta)_\Z  \setminus  \{ \v 0 \}} \left( \deg(\v u)-1 \right).
\] 

\begin{definition} We say that the $k$-th Hasse--Witt condition holds when 
\be\label{HW-condition-s}
hw^{(\ell)}:=p^{- L(\ell)} \det (HW^{(\ell})) \in \Z_p\lb t \rb^\times
\ee
for each $1 \le \ell \le k$. 
\end{definition}

The present section is devoted to proving the following theorem. We say that a proper face $\tau$
of $\Delta$ is \emph{a simplex of volume~$1$} if it has
$\dim(\tau)+1$ vertices and all lattice points in the $\R_{\ge0}$-cone generated by $\tau$
are integer linear combinations of the vertex vectors.

\begin{theorem}\label{HWC}
Suppose $p>n$. If all proper faces of $\Delta$ are simplices of volume~1 and 
all coefficients of $g(\v x)$ are in $\Z_p^\times$ then the $k$-th Hasse--Witt condition
holds for any $k<p$.  
\end{theorem} 
 
One easily verifies that the condition of the theorem is fullfilled by the simplicial
and octahedral Newton polytopes $\Delta$. 
The proof will be given at the end of this section. We start with some preparations.
Note that for $\v w \in \R^n$ one has $\deg \v w =\max_\tau \ell_{\tau}(\v w)$, where
the maximum is taken over all codimension one faces $\tau\subset \Delta$ and $\ell_\tau$
is the linear functional which takes value $1$ on $\tau$.
We will use the following properties of the degree function:

\begin{lemma}\label{deg-props} Let $\v w,\v w_1,\v w_2 \in \R^n$. For a face $\tau$ of any dimension
we denote by $C(\tau)$ the positive $\R_{\ge0}$-cone spanned by $\tau$. Then,
\begin{itemize}
\item[(a)] for any codimension~1 face $\tau$ we have $\deg(\v w)=l_\tau(\v w) \Leftrightarrow \v w \in C(\tau)$ 
\item[(b)] $\deg(\v w_1 + \v w_2) \le \deg \v w_1 + \deg \v w_1$
\item[(c)] the condition $\deg(\v w_1 + \v w_2) = \deg \v w_1 + \deg \v w_1$ is equivalent to the fact that  for all proper faces $\tau$ one has: $\v w_1 + \v w_2 \in C(\tau) \Rightarrow \v w_1, \v w_2 \in C(\tau)$. 
\end{itemize}
\end{lemma}
\begin{proof} (a) Let us assume $\v w \ne 0$ and set $\lambda = \deg \v w$. It follows from the definition of $\deg$ that $\v w / \lambda \in \Delta$. Suppose that $\deg \v w = \ell_\tau(\v w)$. Then $1 = \deg(\v w/\lambda) = \ell\tau(\v w/\lambda)$.  Since $\v w/\lambda \in \Delta$ we see that it must lie in the face $\tau$. 
Hence $\v w \in \lambda \tau \subset C(\tau)$.  Conversely, if $\v w \in C(\tau)$, choose $\lambda'$ and $\v w' \in \tau$ such that $\v w = \lambda' \v w'$. For all codimension one faces $\tau'$ we have $\ell_{\tau'}(\v w') \le \ell_\tau (\v w') = 1$. After multiplication by $\lambda'$ we get $\ell_{\tau'}(\v w) \le \ell_\tau(\v w) = \lambda'$. Hence $\deg \v w = \lambda'$. 
 
(b) Choose $\tau$ such that $\ell_\tau(\v w_1 + \v w_2) = \deg(\v w_1 + \v w_2)$. Then,
 \[
\deg(\v w_1 +\v w_2 ) =\ell_\tau (\v w_1 + \v w_2 ) = \ell_\tau(\v w_1 ) + \ell_\tau (\v w_2 ) \le \deg \v w_1 + \deg \v w_2.
\]
(c)``$\Rightarrow$'' It suffices to prove the condition on the right side for all codimension one faces $\tau$. Choose $\tau$ of codimension~1 such that $\v w_1+\v w_2 \in C(\tau)$. Then, by (a), $\deg(\v w_1 +\v w_2 ) =\ell_\tau (\v w_1 + \v w_2 ) = \ell_\tau(\v w_1 ) + \ell_\tau (\v w_2 ) \le \deg \v w_1+ \deg \v w_2$. It is also given that $\deg(\v w_1 + \v w_2) = \deg \v w_1 + \deg \v w_1 $. Hence we must conclude that $\deg \v w_1=\ell_\tau(\v w_1)$
and $\deg \v w_2 =\ell_\tau(\v w_2)$. Hence, again by (a), $\v w_1, \v w_2 \in C(\tau)$. 

``$\Leftarrow$'' Take any codimension~1 face $\tau$ such that $\v w_1+\v w_2 \in C(\tau)$. Then $\v w_1, \v w_2 \in C(\tau)$and it follows from (a) that $\deg(\v w_1 +\v w_2 ) =\ell_\tau (\v w_1 + \v w_2 ) = \ell_\tau(\v w_1 ) + \ell_\tau (\v w_2 ) = \deg \v w_1 + \deg \v w_2$.
\end{proof}

For a subset $S \subset \Delta$ we denote by $HW^{(k)}(S)$ the square submatrix of $HW^{(k)}$ indexed by $\v u, \v v \in (kS)_\Z$. For a proper face $\tau \subsetneq \Delta$ we denote by $P(\tau)$ the pyramid over $\tau$, that is the convex hull of $\tau$ and $\v 0$. The interior of $\tau$ is denoted by $\tau^\circ$, this is the set of points of $\tau$ which do not lie in a subface of $\tau$, except when $\tau$ is a vertex. In that case we take the convention $\tau^\circ = \tau$. In what follows we will denote 
\[
P(\tau^\circ) = \text{ the pyramid over } \tau^\circ \text{ with } \v 0 \text{ removed.}
\]
One can obtain this set from the polytope $P(\tau)$ by removing all its proper faces which contain $\v 0$. 
We remark that the (disjoint) union of $P(\tau^\circ)$ over all proper subfaces $\tau$ of $\Delta$
is precisely $\R^n\setminus\{\v 0\}$.

\begin{proposition}\label{det-HW-decomposition-at-0} With notations as above, one has
\[
\det HW^{(k)}|_{t=0} = \prod_{\tau \subsetneq \Delta} \det HW^{(k)}(P(\tau^\circ))|_{t=0},
\]
where the product is over all proper subfaces $\tau$ of $\Delta$. By convention, when $(k P(\tau^\circ))_\Z=\emptyset$ the respective factor for $\tau$ in the above product is $1$.
\end{proposition}

\begin{proof}We show that the matrix $H:=HW^{(k)}|_{t=0}$ has an upper triangular block structure where
the blocks are indexed by $\v 0$ and the points in $k P(\tau^\circ)$ for all proper subfaces $\tau$ of $\Delta$. We assume the indexes ordered according to increasing dimension of $\tau$, starting with $\v 0$.

Choose $\v v, \v u \in (k \Delta)_\Z$. Let us first deal with $\v v=\v 0$. The coefficient of $\v x^{-\v u}$ in $F^{(k)}$ divided by $t^{-\deg \v u}$ becomes $0$ after setting $t = 0$, unless $\v u=\v 0$. So the matrix column of $H$ corresponding to $\v v=\v 0$ consists of zeroes, except for the entry $H_{\v 0, \v 0}=F^{(k)}|_{t=0}$.
Using \eqref{HW-k-polynomial} and $f=1-tg$ we easily find that $H_{\v 0,\v 0}=1$. 

In general, let $\tau$ be the face such that $\v v$ lies in $P(\tau^\circ)$. The entry $H_{\v u, \v v}$ can be non-zero only if $\deg(p \v v - \v u) = p\deg(\v v) - \deg(\v u)$. Application of Lemma~\ref{deg-props}(c)   with $\v w_1=\v u$ and $\v w_2=p \v v - \v u$ shows that $\v u \in C(\tau)$. Hence either $\v u \in k P(\tau^\circ)$, or $\v u \in k P((\tau')^\circ)$ for some proper subface $\tau' \subset \tau$, or $\v u = \v 0$.  This gives us the  upper triangular block structure and the factorization of the determinant follows.
\end{proof}

In the following lemma we use the assumption $t^\sigma/t^p \in 1+pt\Z_p\lb t \rb$.

\begin{lemma}\label{pyramid-F-HW-matrix-restriction} Let $\tau$ be a proper face of $\Delta$. For $\v u, \v v \in (k P(\tau^\circ))_\Z$ the $(\v u, \v v)$-entry in the matrix $HW^{(k)}(P(\tau^\circ))|_{t=0}$
equals the coefficient of $\v x^{p \v v - \v u}$ in the polynomial  
\be\label{HW-k-polynomial-to-C(F)-0}
F^{(k)}_{h}(\v x) = h(\v x)^{p-k} \sum_{r=0}^{k-1} (h(\v x^p) - h(\v x)^p)^r h(\v x^p)^{k-1-r},
\ee
where $h(\v x) = 1- \sum_{\v w \in \tau_\Z} g_\v w \, \v x^\v w$. Here $g_{\v w}$ is the coefficient
of $\v x^{\v w}$ in $g(\v x)$.
\end{lemma}

\begin{proof}
Write $\ell_\tau(\v w)=\sum_{i=1}^n \ell_i w_i$ with $\ell_i \in \Z$. Let us replace $x_i$ by 
$x_i/t^{\ell_i}$ for $i=1,\ldots,n$. Observe that the monomial $\v x^\v w$ gets replaced by
$\v x^\v w / t^{\ell_\tau(\v w)} = t^\delta \v x^\v w / t^{\deg \v w}$, where $\delta \ge 0$
and $\delta = 0$ if and only if $\v w \in C(\tau)$ as a consequence of Lemma~\ref{deg-props}(a).
Carrying out this substitution in the polynomials $f(\v x)$, $f^\sigma(\v x^p)$
and~\eqref{HW-k-polynomial} and setting $t=0$ we obtain respectively $h(\v x)$, $h(\v x^p)$ 
and~\eqref{HW-k-polynomial-to-C(F)-0}. Our claim immediately follows from this. 
\end{proof}

From now on we will consider the situation when the proper subfaces of $\Delta$ are simplices
of volume $1$.

\begin{proposition}\label{det-HW-pyramid-p-adic-order} Let $\tau$ be a simplicial face of volume $1$ spanned by vertices $\v v_1, \ldots, \v v_s$. Let $h(\v x) = 1- \sum_{i=1}^s h_i \v x^{\v v_i}$ be a Laurent polynomial with $h_i \in \Z_p^\times$ for $0 \le i \le s$. 

Then the set $(k P(\tau^\circ))_\Z$ is empty when $k < s$. When $k \ge s$ let $H$ be the matrix indexed by $\v u, \v v \in (k P(\tau^\circ))_\Z$, with entries given by the coefficient of $\v x^{p \v v - \v u}$ in the polynomial $F^{(k)}_h(\v x)$ given by formula~\eqref{HW-k-polynomial-to-C(F)-0}. Then the $p$-adic order of $\det(H)$ equals
\[
\sum_{\v u \in (k P(\tau^\circ))_\Z} \left(\deg(\v u) - 1 \right).
\]
\end{proposition}

\begin{proof} 
Let us introduce the coordinates $y_i=\v x^{\v v_i}$ for $i=1,\ldots,s$. 
Since $\v v_1, \ldots, \v v_s$
are generators of a volume 1 simplex, every $\v x^{\v w}$ with $\v w\in C(\tau^\circ)$ can be written as a
a monomial in $y_1,\ldots,y_s$ (with exponents $>0$). Thus we see that
we can restrict ourselves to a polynomial $\tilde{h}(\v y)=1-h_1y_1-\cdots-h_sy_s$. Then
$F^{(k)}_{\tilde{h}(\v y)}=F^{(k)}_{h(\v x)}$. Furthermore, the degree
of the monomial $\v x^{\v u}$ is now simply the usual total degree of the corresponding monomial in
$y_1,\ldots,y_s$. The monomials $\v x^{\v u}\in(kP(\tau^\circ))_\Z$ simply become the monomials
$\v y^{\v w}$ with $\v w\in Q_k$ where $Q_k:=\{\v w|w_1,\ldots,w_s>0,w_1+\cdots+w_s\le k\}$. 
In particular $Q_k$ is empty $k<s$, so we see that 
$H_{\tilde{h}(\v y)}^{(k)}=1$ if $k<s$, as asserted in our proposition.

Now assume that $k\ge s$. Observe that 
\[
F^{(k)}_{\tilde{h}(\v y)}=\frac{(\tilde{h}(\v y^p)-\tilde{h}(\v y)^p)^k-\tilde{h}(\v y^p)^k}{-\tilde{h}(\v y)^k}
\is\frac{\tilde{h}(\v y^p)^k}{\tilde{h}(\v y)^k}\is\frac{(1-h_1y_1^p-\cdots-h_sy_s^p)^k}
{(1-h_1y_1-\cdots-h_sy_s)^k}\mod{p^k}.
\]
Expand the right hand side as power series in 
$y_1,\ldots,y_s$.
\be\label{HW-polynomial-cone-y}
\frac{\tilde{h}(\v y^p)^k}{\tilde{h}(\v y)^k}=\sum_{j=0}^k(-1)^j\binom{k}{j}(h_1y_1^p+\cdots+h_sy_s^p)^j
\sum_{i=0}^k\binom{i+k-1}{k-1}(h_1y_1+\cdots+h_sy_s)^i.
\ee
Let $H'\is H\mod{p^k}$ be the matrix with entries 
\[
H'_{\v u,\v v}=\text{coefficient of $\v y^{p\v u-\v v}$ in \eqref{HW-polynomial-cone-y}},
\]
where $\v u,\v v\in Q_k$.

In what follows we will compute the Smith normal form of $H'$. To every $r \times r$ matrix $M$ with coeffcients in $\Z_p$ there exist matrices $S,T\in {\rm GL}_r(\Z_p)$ such that $S M T = {\rm diag}(p^{k_1},\ldots,p^{k_s})$ with $k_1 \le k_2 \le \ldots \le k_r$. This is the Smith normal form of $M$. The entries $p^{k_i}$ are uniquely determined and known as elementary divisors of $M$.

Let us make an observation. Suppose that the elementary divisors of $M$ satisfy $k_\ell < k$ for all $\ell$ and $H \is M \mod{p^k}$. Then the $p$-adic order of $\det(H)$ equals $\sum_{\ell=1}^r k_\ell$. To prove this fact consider $S H T = {\rm diag}(p^{k_1},\ldots,p^{k_s}) + O(p^k)$ and multiply on both sides by the inverse of the Smith normal matrix. Since $k_{\ell}<k$ for all $\ell$ we get 
${\rm diag}(p^{k_1},\ldots,p^{k_s})^{-1}SHT = Id + O(p)$ and take the determinant on both sides. 

We will show that the elementary divisors of $H'$ are given by
\be\label{elementary-divisors}
p^{|\v u|-1}, \qquad \v u \in Q_k,
\ee
where $|\v u|=u_1+\cdots+u_s$.
The statement of the proposition then follows from this fact due to the observation made above. 

Consider $\v u,\v v\in Q_k$. The entry $H'_{\v u, \v v}$ is the coefficient of $\v y^{p \v v - \v u}$ in the 
right-hand side of~\eqref{HW-polynomial-cone-y}. The contributions to this coefficient come from products
of the form $(\sum_rh_ry_r^p)^j(\sum_rh_ry_r)^i$, where $pj+i=p\deg\v v-\deg\v u$. Take a term
with $\v y^{p\v w}$ in $(\sum_rh_ry_r^p)^j$ and a term with $\v y^{\v w'}$ in $(\sum_rh_ry_r)^i$. 
Then $p\v w+\v w'=p\v v-\v u$, hence $\v w'+\v u\is0\mod{p}$. Since $\v w'+\v u$ has positive entries 
there exists an integer vector $\v v'$ with positive entries such that $\v w'+\v u=p\v v'$, hence 
$\v v'=\v v-\v w$. Since $\v v\in Q_k$ we conclude that $\v v'\in Q_k$. Moreover, $\v w=\v v-\v v'$
and $\v w'=p\v v'-\v u$. 
Consequently, $H'_{\v u,\v v}$ equals
the sum of the products
\[\bal
&\text{coeff of $\v y^{p(\v v -\v v')}$ in} (-1)^{|\v v-\v v'|}\binom{k}{|\v v-\v v'|}
\left(\sum_rh_ry_r^p\right)^{|\v v-\v v'|}\\
&\times \text{coeff of $\v y^{(p\v v'-\v u)}$ in}\binom{|p\v v'-\v u|+k-1}{k-1}
\left(\sum_rh_ry_r\right)^{|p\v v'-\v u|}
\eal\]
taken over $\v v'\in Q_k$. The coefficient of $\v y^\v z$ in $(\sum_rh_ry_r)^{|\v z|}$ equals
$\v h^{\v z}$ times
\[
\binom{|\v z|}{\v z} = \frac{|\v z|!}{z_1! \dots z_s!},
\]
where $\v h^{\v z}=h_1^{z_1}\cdots h_s^{z_s}$. 
In the following we use the convention that $\binom{|\v z|}{\v z}=0$ if one of the components
of $\v z$ is negative.

We obtain the expression
\[\bal
H'_{\v u, \v v} = \v h^{p\v v-\v u}\sum_{\v v'\in Q_k}&  (-1)^{|\v v-\v v'|}\binom{k}{|\v v-\v v'|} 
\binom{|\v v-\v v'|}{\v v-\v v'} \\
&\times \binom{|p\v v' -\v u| + k - 1}{k-1} \binom{|p\v v' - \v u|}{p \v v' - \v u}
\eal\] 
Therefore we can write $H' = G^{-1}MNG^p$
where all matrices have the same indexing set $Q_k$, $G={\rm diag}(\v h^{\v u})_{\v u\in Q_k}$, and
\[
M_{\v u, \v v} = \binom{|p\v v -\v u|+ k - 1}{k-1} \binom{p|\v v - \v u|}{p \v v - \v u} 
\]
and $N_{\v u, \v v} = (-1)^{|\v v-\v u|}\binom{k}{|\v v-\v u|} \binom{|\v v-\v u|}{\v v-\v u}$.
If we order indices with
non-decreasing degree then $N$ will be upper triangular with $1$'s on the diagonal, so $N$ is also invertible.
Therefore the elementary divisors of $H'$ are the same as of $M$. To compute them let us introduce 
\[
\phi_\v v = \frac1{(k-1)!} \frac{(p|\v v|-1)!}{(p v_1-1)! \ldots (p v_s-1)!}, \quad \v v \in Q_k.
\]
Since $p>k$ and $|\v v| \le k$, one easily sees that $\phi_\v v$ is exactly divisible by $p^{s-1}$. We have
\be\label{M-divided}
M_{\v u, \v v} / \phi_\v v = (p|\v v|)_{k  - |\v u|} \prod_{i=1}^s [p v_i - 1]_{u_i - 1},
\ee
where 
\[
(x)_n = x(x+1)\dots(x+n-1)
\]
is the usual Pochhammer symbol and 
\[
[x]_n = x(x-1)\dots(x-n+1)
\]

is its descending version. Denote by $M'_{\v u, \v v}$ the expression in~\eqref{M-divided}. 
We have $M = M' \; {\rm diag}(\phi_\v v)$. The entry $M'_{\v u, \v v}$ is polynomial of total degree $k-s$ 
in $p v_1, \ldots, p v_s$. Even better, as $\v u$ runs over $Q_k$ these polynomials form a basis in the 
$\Z_p$-module of all polynomials of total degree $\le k-s$. This assertion is proved below in
Lemma~\ref{pochhammer-basis}(a). Therefore there exists a matrix $S$, invertible in $\Z_p$,
such that each column of $M'' = S M'$ consist of all monomials of degree $\le k-s$ in $p v_1, \ldots, p v_s$ in the same ordering: 
\[
M''_{\v w, \v v} = (p \v v)^{\v w} = p^{\deg \v w} \v v ^\v w,
\]
where $\v w$ runs over the integer with positive coordinates and $|\v w| \le k-s$.
We note that this set is in a bijective correspondence with $Q_k$ via the map
\[ 
\v w \to \v u = \v w + (1,\ldots,1).
\]
Lemma~\ref{pochhammer-basis}(b) shows that the square matrix $(\v v^\v w)$ is invertible over $\Z_p$. 
Therefore the elementary divisors of $M''$ and $M'$ are given by $p^{\v w|}$ and the elementary divisors of 
$M$ and $H'$ are $p^{|\v w| + s-1}=p^{|\v u| - 1}$, as was claimed in~\eqref{elementary-divisors}.

\end{proof}

We now prove two auxiliary statements which were used in the last paragraph of the above proof.

\begin{lemma}\label{pochhammer-basis} Let $s \ge 1$ and $m \ge 0$.
\begin{itemize} 
\item[(a)]  Consider the set of polynomials in $s$ variables $\v z = (z_1,\ldots, z_s)$ of total degree~$m$ defined as
\[
P_{m,\v w}(\v z) = (z_1+\ldots+z_s)_{m  - \sum_{i=1}^s w_i} \prod_{i=1}^s [z_i - 1]_{w_i}, \quad 
\]
where $\v w = (w_1,\ldots,w_s)$ runs over the set of vectors 
\[
\v w \in \Z_{\ge 0}^s \text{  with } \sum_{i=1}^s w_i \le m.
\]
When $p \ge m+s$ these polynomials form a basis in the $\Z_p$-module of all polynomials of total degree~$\le m$. 
\item[(b)] Consider the square matrix $(\v v^{\v w})$, where rows are indexed by $\v w$ running over the same set as in part (a) and columns are indexed by
\[
\v v \in \Z_{\ge 1}^s \text{  with } \sum_{i=1}^s v_i \le m+s.
\]
When $p > m$ one has $\det(\v v^\v w) \in \Z_p^\times$. 
\end{itemize}
\end{lemma}
\begin{proof} (a) We fix $p$ and $s$ and prove the statement by induction in $m$. When $m=0$ we have one polynomial $P_{0,\v 0}(\v z)=1$ and the claim is obvious. Let $m>0$ and suppose that the statement is true for $m-1$. Let $A$ denote the $\Z_p$-module spanned by $P_{m,\v w}(\v z)$ for all $\v w$ with $\sum_{i=1}^mw_i \le m$. Let $\v e_1,\ldots,\v e_s$ be the standard basis vectors in $\Z^s$. Take any $\v w$ with $\sum_{i=1}^mw_i < m$ and note that 
\[\bal
\sum_{i=1}^s P_{m, \v w+\v e_i}(\v z) &=\left( \left(z_1+\ldots+z_s\right)_{m - 1 - \sum_{i=1}^sw_i} \prod_{i=1}^s [z_i - 1]_{w_i} \right) \sum_{i=1}^s (z_i - w_i-1) \\
& = P_{m,\v w}(\v z) -  \left(m+s-1\right) P_{m-1,\v w}(\v z).\\
\eal\]
Since $0<m+s-1<p$ we deduce from the above formula that polynomial $P_{m-1,\v w}(\v z)$ belongs to $A$. Now it follows from our inductional assumption that $\v z^\v w \in A$ for any $\v w$ with $\sum_{i=1}^{s} w_i < m$. It remains to consider monomials $\v z^\v w$ with $\sum_{i=1}^{s} w_i = m$. For such $\v w$ we have
\[
P_{m,\v w}(\v z) = \prod_{i=1}^s [z_i - 1]_{w_i} = \v z^{\v w} + \text{ terms of total degree } < m.
\] 
Since the terms of smaller degree belong to $A$, we have $\v z^\v w \in A$. This finishes the proof of the step of induction.

(b) When $\v w$ runs over the indexing set from part (a), polynomials 
\[
Q_{\v w}(\v z) = \prod_{i=1}^s [z_i - 1]_{w_i} = \v z^\v w + \text{ terms of smaller total degree }
\]
form a basis in the $\Z_p$-module of all polynomails of total degree $\le m$. The transition matrix from the monomial basis to this basis has determinant~1, and therefore $\det(\v v^\v w) = \det(Q_\v w(\v v))$. Note that $Q_{\v w}(\v v)=0$ when $w_i \ge v_i$ for some $i$. Let $\v e=(1,\ldots,1)$. We then have $Q_{\v w}(\v v)=0$ when $\sum_{i=1}^s w_i > \sum_{i=1}^s (v_i - 1)$ or $\sum_{i=1}^s w_i = \sum_{i=1}^s (v_i - 1)$ but $\v w \ne \v v - \v e$. If we arrange both indexing sets so that the sum of coordinates is non-decreasing, then
\[
\det(Q_\v w(\v v)) = \pm \prod_{\v v} Q_{\v v-\v e}(\v v) = \pm \prod_{\v v} \prod_{i=1}^s (v_i-1)!.
\] 
In the set indexing columns the maximal value of $v_i-1$ is $m$, hence the above determinant is not divisible by $p$ when $p>m$.  
\end{proof}

\begin{proof}[Proof of Theorem \ref{HWC}]
Let $k\le n$. We use Proposition~\ref{det-HW-decomposition-at-0} to decompose
$\det(HW^{(k)})|_{t=0}$ into a product 
the Hasse--Witt determinants corresponding to $P(\tau^\circ)$ for proper faces $\tau\subset \Delta$
of all dimensions.
Let $\tau$ be a proper face. By Lemma~\ref{pyramid-F-HW-matrix-restriction} the Hasse--Witt matrix 
$H=HW^{(k)}(P(F^\circ)|_{t=0}$ can be computed from the coefficients of the polynomial 
$F^{(k)}_h(\v x)$ with $h=1-\sum_{\v w \in F_\Z} g_\v w \v x^\v w$. 
Since ${\rm vol}(F)=1$ and all $g_\v w \in \Z_p^\times$ we can apply 
Proposition~\ref{det-HW-pyramid-p-adic-order} and conclude that the $p$-adic order of $\det H$
equals $\sum_{\v u \in (k P(F^\circ))_\Z} (\deg \v u-1)$. It remains to note that
$\Delta \setminus \{\v 0\}$ is the disjoint union of sets $ P(F^\circ)$ for all proper faces $F$.
Multiplying by $k$ and taking sets of integral points, we find that $(k \Delta)_\Z \setminus \{\v 0\}$
is the disjoint union of sets $(kP( F^\circ))_\Z$. Hence
\[
\ord_p \det(HW^{(k)})|_{t=0} = \sum_{F \subsetneq \Delta} \sum_{\v u \in (k P(F^\circ))_\Z} (\deg \v u-1) = \sum_{\v u \in (k \Delta)_\Z \setminus \{\v 0\}} (\deg \v u-1) = L(k).
\]
It follows that $p^{-L(k)} \det(HW^{(k)}) \in \Z_p\lb t \rb^\times$. Since $k < p$ is arbitrary,
our claim follows.
\end{proof}

\begin{remark} The fact that $f(\v x)=1-t g(\v x)$ is not essential in Theorem~\ref{HWC}. With little modifications, a similar claim can be proved for any admissible $f(\v x) = \sum_\v u f_\v u \v x^\v u$ such that $f_\v 0|_{t=0} \in \Z_p^\times$ and $(f_\v u/t) |_{t=0} \in \Z_p^\times$ for all $\v u \ne 0$.
\end{remark}

\section{Construction of the Frobenius structure}\label{firstproofs}

Now we turn to the proof of Proposition~\ref{cartier2frobenius}. Let us recall our basic setup that we 
introduced starting on page~\pageref{setup}. We have a Laurent polynomial $f(\v x)=1-t g(\v x)$ with $g \in \Z[x_1^{\pm 1},\ldots,x_n^{\pm 1}]$ such that its Newton polytope $\Delta \subset \R^n$ is reflexive. We start with the ring $R=\Z[t,1/D_f(t)]$ where $D_f(t)\in\Z[t]$ is some polynomial with $D_f(0)\ne0$. The $R$-modules of admissible rational functions $\sO_f$ and $\sO_f^\circ$ consist of elements of the form $(k-1)! A(\v x)/f(\v x)^k$ where $k \ge 1$ and $A$ is an admissible polynomial with coefficients in $R$ and support in $k \Delta$ and $k \Delta^\circ$ repectively. Recall that $d\sO_f$ is a submodule generated by $\theta_i\omega$ for $\omega \in \sO_f$ and $i=1,\ldots,n$. We also have a finite group $\sG$ of monomial substitutions which preserve $f(\v x)$. 

\begin{definition}\label{cyclicMUM}
We shall say that $M = (\sO_f^\circ)^\sG /d \sO_f$ is a cyclic $\theta$-module
of rank $n$ if 
\begin{itemize}
\item[(i)]
for every $m\ge0$ and for
every $\sG$-invariant admissible Laurent polynomial $A(\v x)$
supported in $m\Delta$ we have
\[
m!\frac{A(\v x)}{f(\v x)^{m+1}}\is\sum_{j=0}^{\min(m,n-1)}b_j(t)\theta^j(1/f)
\quad \mod{d \sO_{f}} 
\]
with $b_j(t)\in\Z[t,1/D_f(t)]$ for all $j$. 
\item[(ii)]
The monic order $n$ differential operator $L$ such that $L(1/f)\in d\sO_f$ is irreducible in 
$\Q(t)[\theta]$.
\end{itemize}

Note that existence of $L$ in~(ii) follows from~(i). We call $L$ the Picard-Fuchs differential operator associated to $M$. Write $L=\theta^n + \sum_{j=0}^{n-1}a_j(t)\theta^j$ with $a_j(t)\in\Z[t,1/D_f(t)]$
for all $j$. 
We shall say that $M$ is of MUM-type (maximally unipotent local monodromy) if $a_j(0)=0$
for all $j<n$.
\end{definition}

For any prime $p$ not dividing $\#\sG\times D_f(0)$ and $p\ge n$ we embed $\Z[t,1/D_f(t)]$ in $\Z_p\lb t \rb$ and fix a Frobenius lift $\sigma$ on $\Z_p\lb t \rb$. We assume it is given by $t^\sigma = t^p(1 + p \, t \, u(t))$ for some $u \in \Z_p\lb t \rb$.

\begin{proposition}\label{Lambda-entries-p-divisibility} 
Let $M = (\sO^\circ_f)^\sG/d \sO_f$ be a cyclic $\theta$-module of rank $n$. Suppose that $p\ge n$.
Then there exist $\lambda_i(t)\in p^i\Z_p\lb t\rb$ such that
\be\label{cartierformula}
\cartier(1/f)\is \sum_{i=0}^{n-1}\lambda_i(t)(\theta^i(1/f))^\sigma
\quad \mod{d\hat\sO_{f^\sigma}}.
\ee
\end{proposition}

\begin{proof} From specialization of \cite[(3)]{BeVlIII} to $m=1,A=1$ we get the following explicit formula
\be\label{f-inverse-expansion}
\cartier(1/f)=\sum_{k\ge 0}\frac{p^k}{k!}\times k!\frac{Q_k(\v x)}{(f^\sigma)^{k+1}},
\ee 
with $Q_k(\v x)=\cartier(G(\v x)^kf(\v x)^{p-1})$ and 
\[
G(\v x)=\frac{1}{p}(f^\sigma(\v x^p)-f(\v x)^p).
\]
The discussion preceding \cite[(3)]{BeVlIII} assures us that the Laurent polynomials $Q_k(\v x)$ are
admissible with respect to $t^\sigma$.  So we can carry out Dwork-Griffiths reduction modulo
$d\sO_{f^\sigma}$ using Definition~\ref{cyclicMUM}(i) with $f$ replaced by $f^\sigma$.
The terms $k!Q_k(\v x)/(f^\sigma)^{k+1}$ 
are equivalent modulo $d\sO_{f^\sigma}$ to a $\Z_p\lb t\rb$-linear combination of
$(\theta^i(1/f))^\sigma,i=0,1,\ldots,n-1$. 
Thus the existence of the $\lambda_i(t)\in \Z_p\lb t \rb$ follows. 

For any $0\le i<n$ the contributions to $\lambda_i(t)$ come from terms in \eqref{f-inverse-expansion} with 
$k\ge i$. Since $p\ge n$ we have that $\ord_p(p^k/k!)\ge i$ if $k\ge i$. 
This implies that $\lambda_i(t)$ is divisible by $p^i$ for all $0\le i<n$.
\end{proof}

\begin{remark} Consider the $p$-adic completion $R=\Z[t,1/D_f(t)]\,\hat \; \subset \Z_p\lb t \rb$. Though we will not use this fact, let us remark that if $t^\sigma \in R$ then $\lambda_i(t)$ constructed in the above proof  also belong to $R$. 
\end{remark}

\begin{remark}
In the proof of Proposition \ref{cartier2frobenius}
we need independence of $\theta^i(1/f)$, $i=0,\ldots,n-1$ modulo $d\hat\sO_f$ over $\Z_p\lb t\rb$.
It should follow from the irreducibility in $\Q(t)[\theta]$ of the Picard-Fuchs operator.
Unfortunately we have seen no easy way do deduce this. Instead we need to invoke the $n$-th
Hasse--Witt assumption and the following lemmas. 
\end{remark}

\begin{lemma}\label{independence-over-Qt}
Let $(\sO^\circ_f)^\sG/d \sO_f$ be a cyclic $\theta$-module of rank $n$.
Then the elements $\theta^i(1/f), i=0,1,\ldots,n-1$ are linearly independent modulo
$d\sO_f$ over $\Q(t)$.
\end{lemma}

\begin{proof} Rational functions $\omega = A(\v x)/f(\v x)^m \in \sO_f$ can be expanded as power series as follows: 
\[
A(\v x)/(1-t g(\v x))^m = A(\v x) \sum_{k=0}^\infty \binom{k+m-1}{m-1} t^m g(\v x)^m = \sum_{\v u \in \Z} c_\v u \v x^\v u
\]
with some coefficients $c_\v u \in \Z_p\lb t \rb$. Note that the constant term $c_\v 0=c_\v 0(\omega)$ vanishes when $\omega \in d\sO_f$. For this reason the map $c_\v 0: \sO_f \to \Z_p\lb t\rb$ was called a \emph{period map}
in \S 2 of Dwork crystals II. Note that this map is $\Z_p\lb t \rb$-linear and commutes with $\theta$. The constant term $c_\v 0(1/f)$ shall be denoted by $F_0(t)$, this is precisely the power series solution to the Picard-Fuchs operator $L$ normalized so that $F_0(0)=1$. Suppose that there exist $b_0,\ldots,b_{n-1} \in \Q(t)$ not all zero such that   
$\sum_{j=0}^{n-1} b_j \theta^j(1/f) \in d\sO_f$. Then $F_0(t)$ is also annihilated by the lower order operator $L'=\sum_{j=0}^{n-1} b_j\theta^j$. Using the right Euclidean algorithm in $\Q(t)[\theta]$ we then find a right common divisor $L''$ of $L$ and $L'$. This operator of smaller order also annihilates $F_0$. Hence $L'' \ne 0$, which contradicts  irreducibility of $L$ in $\Q(t)[\theta]$.
\end{proof}

\begin{lemma}\label{independence-over-powerseries}
Let $(\sO^\circ_f)^\sG/d \sO_f$ be a cyclic $\theta$-module of rank $n$.
Then the elements $\theta^i(1/f), i=0,1,\ldots,n-1$ are linearly independent modulo $d\sO_f$ 
over $\Z_p\lb t\rb$.
\end{lemma}

\begin{proof}
Suppose we have $c_j(t)\in\Z_p\lb t\rb$ such that
\[
\sum_{j=0}^{n-1}c_j(t)\theta^j(1/f)\in d\sO_f.
\] 
Then there exists an integer $k$ and Laurent polynomials $h_i(\v x)$, $i=1,\ldots,n$ with
support in $k\Delta$ and coefficients in $\Z_p\lb t\rb$ such that
\[
\sum_{j=0}^{n-1}c_j(t)\theta^j(1/f)=\sum_{i=1}^nx_i\frac{\partial}{\partial x_i}
\frac{h_i(\v x)}{f(\v x)^k}.
\]
We can rewrite this equality as a system of linear equation in the $c_j$ and the coefficients
of the $h_i$ with coefficients in $\Q(t)$. The existence of a solution in $\Z_p\lb t\rb$
with at least one non-zero $c_j$ implies the existence of a similar solution in $\Q(t)$.
Hence the $\theta^i(1/f)$ would be linearly dependent over $\Q(t)$, contradicting
Lemma \ref{independence-over-Qt}. 
\end{proof}

\begin{remark}
In the same way one can show that the elements $\theta^i(1/f)^\sigma$, $i=0,\ldots,n-1$
are linearly independent modulo $d\sO_{f^\sigma}$ over $\Z_p\lb t\rb$.
\end{remark}

\begin{lemma}\label{independence-mod-hat-Of}
Let $(\sO^\circ_f)^\sG/d \sO_f$ be a cyclic $\theta$-module of rank $n$. Suppose also that
the $n$-th Hasse--Witt condition holds. Then the elements $\theta^i(1/f), i=0,1,\ldots,n-1$ 
are linearly independent modulo $d\hat\sO_f$ over $\Z_p\lb t\rb$.
\end{lemma}

\begin{proof}
Let $c_0, \ldots, c_{n-1} \in \Z_p\lb t \rb$ be such that 
$\sum_{j=0}^{n-1} c_j \theta^j(1/f) \in d\hat\sO_f$. Let us write this element as 
\[
\sum_{i=1}^{n} x_i \frac{\partial}{\partial x_i}\nu_i
\]
for some $\nu_i \in \hat\sO_f$. Using the decomposition~\eqref{O-f-decomposition-k} for $k=n-1$ we write 
$\nu_i = \omega_i + \delta_i$ with $\omega_i \in \sO_f(n-1)$ and $\delta_i \in \fil_{n-1}$ for each $i$.
It follows that 
\[
\sum_{j=0}^{n-1} c_j \theta^j(1/f) - \sum_{i=1}^{n} x_i \frac{\partial}{\partial x_i}
\omega_i \in \sO(n) \cap \fil_n.
\]
As we have the decomposition~\eqref{O-f-decomposition-k} for $k=n$ it follows that this element is $0$.
Hence $\sum_{j=0}^{n-1} c_j \theta^j(1/f) \in d\sO_f$.
Lemma \ref{independence-over-powerseries} then implies that $c_j(t)=0$ for all $j$.
\end{proof}

\begin{remark}\label{independence-mod-hat-Of-sigma}
In the same way, using decompositions~~\eqref{O-f-decomposition-k} in $\hat\sO_{f^\sigma}$,
one can show that the elements $\theta^i(1/f)^\sigma$, $i=0,\ldots,n-1$  are linearly independent
modulo  $d\hat\sO_{f^\sigma}$ over $\Z_p\lb t\rb$. 
\end{remark}

\begin{proof}[Proof of Proposition \ref{cartier2frobenius}]
Let $L$ be the Picard-Fuchs operator associated to $M$. Let $L^\sigma$ be the
operator obtained from $L$ by applying $\sigma$ to its coeffcients and replacing
$\theta$ by $\theta^\sigma=\frac{t^\sigma}{\theta(t^\sigma)}\theta$. We then normalize
so that the coefficient of $\theta^n$ becomes $1$. Clearly the operator $L^\sigma$
annihilates $y_i^\sigma$ for $i=0,\ldots,n-1$ and we have $L^\sigma(1/f^\sigma)\in
d\sO_{f^\sigma}$. 

Consider (\ref{cartierformula}) and replace $(\theta^i(1/f))^\sigma=(\theta^\sigma)^i
(1/f^\sigma)$ by $\left(\frac{t^\sigma}{\theta t^\sigma}\theta\right)^i(1/f^\sigma)$.
Note that 
$\frac{\theta t^\sigma}{t^\sigma}=p(1+u(t))$ for some $u(t)\in t\Z_p\lb t\rb$.
Since $\lambda_j(t)$ is divisible by $p^j$ for all $j$ we can rewrite
(\ref{cartierformula}) as
\be\label{cartier-to-A}
\cartier(1/f)\is\mathcal{A}(1/f^\sigma)\quad \mod{d\hat\sO_{f^\sigma}},
\ee
where
\[
\sA=A_0(t)+A_1(t)\theta+\cdots+A_{n-1}(t)\theta^{n-1}\in\Z_p\lb t\rb[\theta].
\]
We have $A_i(0)=p^{-i}\lambda_i(0)$ for $0 \le i < n$ because
$\left(p\frac{t^\sigma}{\theta t^\sigma}\theta\right)^i \is \theta^i$ modulo 
$t \Z_p\lb t \rb[\theta]$. Apply $L$ to~\eqref{cartier-to-A} from the left and observe that,
by the commutation $\cartier\circ\theta=\theta\circ\cartier$, we have
\[
L(\cartier(1/f))=\cartier(L(1/f))\subset\cartier(d\hat\sO_f)
\subset d\hat\sO_{f^\sigma}.
\]
Hence $L\circ\mathcal{A}(1/f^\sigma)\in d\hat\sO_{f^\sigma}$. Using the right 
Euclidean algorithm in $\Z_p\lb t\rb[\theta]$ we can find differential operators $\mathcal{B}$ and
$\mathcal{N}$ of order less than $n$ such that $L\circ\sA=\mathcal{B}\circ L^\sigma
+\mathcal{N}$. Since $L^\sigma(1/f^\sigma) \in d \sO_{f^\sigma}$ and $\mathcal{B}$ sends 
$d \sO_{f^\sigma}$ to itself, it follows that 
$\mathcal{N}(1/f^\sigma)\in d\hat\sO_{f^\sigma}$.
Using Remark~\ref{independence-mod-hat-Of-sigma} we conclude that
$\mathcal{N}=0$. Hence $L^\sigma$ is a right divisor of $L\circ\sA$. In
particular this implies that for any solution $y(t)$ of $L(y)=0$ the
composition $\mathcal{A}(y^\sigma)$ lies in the kernel of $L$. 

Finally we must show that $A_0(0)=1$. For this we will use the \emph{period map}
$c_\v 0: \hat\sO_{f^\sigma} \to \Z_p\lb t \rb$ which was defined in the proof of 
Lemma~\ref{independence-over-Qt} by taking the constant term of the $t$-adically convergent formal expansion. This map is $\Z_p\lb t \rb$-linear, commutes with $\theta$ and vanishes on $d \hat\sO_{f^\sigma}$.
Applying the period map 
$c_\v 0$ to the identity $\cartier(1/f)\is\sA(1/f^\sigma)
\mod{d\hat\sO_{f^\sigma}}$ we get $F_0(t)=\sA(F_0^\sigma)$, where $F_0(t)=c_\v 0(1/f)$ is a power series
with $F_0(0)=1$. Setting $t=0$ on both sides gives us
$1=A_0(0)$, as desired.
\end{proof}

\section{Simplicial family}\label{sec:simplicial}
Let 
\[
f(\v x)=1-t\left(x_1+\cdots+x_n+\frac{1}{x_1\cdots x_n}\right).
\]
\begin{proposition}\label{MUMsimplicial}
Let $D_f(t)=(n+1)(1-((n+1)t)^{n+1})$. Then the $\Z[t,1/D_f(t)]$-module $\sO_f^\circ/d\sO_f$
is a cyclic $\theta$-module of MUM-type. The Picard-Fuchs operator reads
\[
L=\theta^n-((n+1)t)^{n+1}(\theta+1)\cdots(\theta+n).
\]
\end{proposition}

Note that if we take the constant term (in $x_1,\ldots,x_n$) in each summand 
of the series expansion $1/f(\v x)=\sum_{k\ge0}t^kg(\v x)^k$ we get in the simplicial case the power series
\[
\sum_{k\ge0}\frac{((n+1)k)!}{(k!)^{n+1}},
\] 
which is the power series solution of our Picard-Fuchs equation $L(y)=0$.

Notice that in the following proof we do not have to invoke the symmetry group of $g(\v x)$, which was denoted in the Introduction by $\sG \cong S_{n+1}$. Proposition~\ref{MUMsimplicial} tells us that the module $\sO_f^\circ/d\sO_f$ is generated by $\sG$-symmetric functions, and therefore we can conclude that $(\sO_f^\circ)^\sG/d\sO_f \cong \sO_f^\circ/d\sO_f$.

\begin{proof}[Proof of Proposition \ref{MUMsimplicial}]
In order to make the symmetric background of $g(\v x)$
more manifest we introduce $x_0=1/(x_1\cdots x_n)$ and write any monomial $\v x^{\v u}$ in
the form $\v x^{\v w}=x_0^{w_0}x_1^{w_1}\cdots x_n^{w_n}$ with $w_0,\ldots,w_n\in\Z_{\ge0}$. 
Of course 
the exponent vector $(w_0,\ldots,w_n)$ is only determined up to a shift by integer multiples of $(1,\ldots,1)$.
Let us write $|\v w|=w_0+\cdots+w_n$. 
Then $\deg(\v x^{\v w})\le|\v w|$ and we have equality if and only if $\min_iw_i=0$.
We also rewrite $f(\v x)$ as $1-t(x_0+\cdots+x_n)$. Observe that for $i=1,\ldots,n$,
\be\label{i-derivative}
\theta_i\frac{\v x^{\v w}}{f(\v x)^m}=(w_i-w_0)\frac{\v x^{\v w}}{f(\v x)^m}+
m(x_i-x_0)\frac{\v x^{\v w}}{f(\v x)^{m+1}}.
\ee
Take the sum over $i=1,\ldots,n$ and multiply by $(m-1)!t^{|\v w|}$. 
\[
0\is (|\v w|-(n+1)w_0)(m-1)!\frac{t^{|\v w|}\v x^{\v w}}{f(\v x)^m}+t(x_0+\cdots+x_n-(n+1)x_0)m!
\frac{t^{|\v w|}\v x^{\v w}}{f(\v x)^{m+1}}\mod{d\sO_f}.
\]
Now replace $t(x_0+\cdots+x_n)$ by $1-f(\v x)$ and we get, after rearranging terms,
\be\label{reduction-x-0}
m!\frac{t^{|\v w|+1}\v x^{\v w}x_0}{f(\v x)^{m+1}}\is m!\frac{1}{n+1}\frac{t^{|\v w|}\v x^{\v w}}{f(\v x)^{m+1}}
+(m-1)!\left(\frac{|\v w|-m}{n+1}-w_0\right)\frac{t^{|\v w|}\v x^{\v w}}{f(\v x)^m}\mod{d\sO_f}.
\ee
Combining this with \eqref{i-derivative} we find
\be\label{reduction-x-i}
m!\frac{t^{|\v w|+1}\v x^{\v w}x_i}{f(\v x)^{m+1}}\is m!\frac{1}{n+1}\frac{t^{|\v w|}\v x^{\v w}}{f(\v x)^{m+1}}
+(m-1)!\left(\frac{|\v w|-m}{n+1}-w_i\right)\frac{t^{|\v w|}\v x^{\v w}}{f(\v x)^m}\mod{d\sO_f}
\ee
for $i=0,1,\ldots,n$. Using these reduction formulae one easily sees by induction that any term
$m!\frac{t^{\deg(\v u)}  \v x^\v u}{f(\v x)^{m+1}}$ with $\deg(\v u) \le m$ is equivalent modulo
$d \sO_f$ to a $\Z[(n+1)^{-1}]$-linear combination of $\frac{k!}{f(\v x)^{k+1}}$ with $k = 0,\ldots,m$.

To finish the verification of part (i) of Definition \ref{cyclicMUM} we must write $\theta^n(1/f)$ as a
$\Z[t,1/D_f(t)]$-
linear combination of $\theta^i(1/f)$, $i=0,1,\ldots,n-1$. To that end we observe that $x_0x_1\cdots x_n=1$.
So 
\be\label{first-reduction}
n!\frac{t^{n+1}x_0\cdots x_n}{f(\v x)^{n+1}}=n!\frac{t^{n+1}}{f(\v x)^{n+1}}=t^{n+1}
(\theta+1)\cdots(\theta+n)\left(\frac{1}{f(\v x)}\right).
\ee
We now show that the left hand side equals $(n+1)^{-n-1}\theta^n(1/f(\v x))$ modulo $d\sO_f$.
Set $i=n,m=n,\v x^{\v w}=x_0x_1\cdots x_{n-1}$ in \eqref{reduction-x-i}. We
find
\be\label{second-reduction}
n!\frac{t^{n+1}x_0\cdots x_n}{f(\v x)^{n+1}}\is n!\frac{1}{n+1}\frac{t^{n}x_0\cdots x_{n-1}}{f(\v x)^{n+1}}
\mod{d\sO_f}.
\ee
Choose any $1\le m\le n$ and set $i=m-1,\v x^{\v w}=x_0\cdots x_{m-2}$ in \eqref{reduction-x-i}. If $m=1$
we take $\v x^{\v w}=1$. We get
\begin{eqnarray*}
m!\frac{t^{m}x_0\cdots x_{m-1}}{f(\v x)^{m+1}}&\is&
\frac{1}{n+1}\left(m!\frac{t^{m-1}x_0\cdots x_{m-2}}{f(\v x)^{m+1}}
-(m-1)!\frac{t^{m-1}x_0\cdots x_{m-2}}{f(\v x)^m}\right)\mod{d\sO_f}\\
&\is& \frac{1}{n+1}\theta\left((m-1)!\frac{t^{m-1}x_0\cdots x_{m-2}}{f(\v x)^m}\right)\mod{d\sO_f}.
\end{eqnarray*}
By recursive application for $m=n, n-1,\ldots,1$ and finally \eqref{second-reduction},\eqref{first-reduction}
we find that 
\[
(n+1)^{-n-1}\theta^n\left(\frac{1}{f(\v x)}\right)\is 
t^{n+1}(\theta+1)\cdots(\theta+n)\left(\frac{1}{f(\v x)}\right)\mod{d\sO_f},
\]
which verifies part (i) of Definition \ref{cyclicMUM}. It also establishes the MUM-property of $L$.

To verify part (ii) we remark that the hypergeometric equation with parameters $\frac{k}{n+1}$,
$k=1,\ldots,n$ and $1,1,\ldots,1$ has the property that the parameter sets are disjoint modulo $\Z$.
This is the well-known irreducibility criterion for hypergeometric differential equations,
see \cite{BH89}.
Our equation $L(y)=0$ arises from this hypergemeometric equation after the substitution 
$t\to((n+1)t)^{n+1}$. The monodromy of the hypergeometric equation is generated by a maximally
unipotent matrix $U$ coming from the monodromy at $0$ and a pseudo reflection $S$ from
the local monodromy at $1$. Then the monodromy group of $L(y)=0$ contains $U^{n+1}$ and $S$.
It is straightforward to see that if $U$ and $S$ generate a matrix group that acts irreducibly,
then the same holds for $U^{n+1}$ and $S$. Hence the monodromy group of $L(y)=0$ is irreducible,
which implies in particular that $L$ is irreducible in $\Q(t)[\theta]$. 
\end{proof}

\section{Hyperoctahedral family}\label{sec:hyperoctahedral}
Let 
\[
f(\v x)=1-t\left(x_1+\frac{1}{x_1}+\cdots+x_n+\frac{1}{x_n}\right)
\]
and let $\sG\cong S_n \times (\Z/2\Z)^n$ be the symmetry group generated by the permutations of
$x_1,\ldots,x_n$ and $x_i\to x_i^{\pm1}$.

\begin{proposition}\label{MUMhyperoctahedral}
There exists $D_f(t)\in n!\Z[t]$ with $D_f(0)=n!$ such that over the ring
$\Z[t,1/D_f(t)]$ the module $(\sO_f^\circ)^{\sG}/d\sO_f$
satisfies property (i) of Definition \ref{cyclicMUM}. There exists a monic operator $L$ of order $n$
with coefficients in this ring and of MUM-type such that $L(1/f) \in d\sO_f$.  
\end{proposition}

\begin{proof}[Proof of Proposition \ref{MUMhyperoctahedral}]
For the beginning we will work over the ring $R=\Z[1/n!][t]$.
Define $X_i=x_i+\frac{1}{x_i}$ and $Y_i=x_i-\frac{1}{x_i}$ for $i=1,\ldots,n$. 
Any element in $\Z[x_1^{\pm 1},\ldots,x_n^{\pm 1}]^\sG$ is a symmetric function in $X_1,\ldots,X_n$.
Notice also that
$\theta_iX_i=Y_i$ and $\theta_iY_i=X_i$. 

We first reduce $m!\frac{A(\v x)}{f(\v x)^{m+1}}$ as a linear combination of $\theta^i(1/f)$
for $i=0,1,\ldots,m$. We shall reduce this problem to the case when $A$ is an
elementary symmetric function of $X_1,\ldots,X_n$. 

For any $i=1,\ldots,n$, and polynomial $B$ in the $X_i$ consider
\begin{eqnarray*}
\theta_i\left(\frac{Y_iB}{f^{m}}\right)&=&\frac{X_iB}{f^m}+\frac{Y_i\theta_iB}{f^m}+mt\frac{Y_i^2B}{f^{m+1}}\\
&=&\frac{X_iB}{f^m}+\frac{Y_i\theta_iB}{f^m}+mt\frac{X_i^2B}{f^{m+1}}-mt\frac{4B}{f^{m+1}}
\end{eqnarray*}
Multiply on both sides by $(m-1)!t^{\deg{B}+1}$. We get
\[
m!t^{\deg{B}+2}\frac{X_i^2B}{f^{m+1}}\is m!t^{\deg{B}+2}\frac{4B}{f^{m+1}}-
(m-1)!t^{\deg{B}+1}\frac{X_iB}{f^m}-(m-1)!t^{\deg{B}+1}\frac{Y_i\theta_iB}{f^m}\mod{d\sO_f}.
\]
So we see that any admissible term of the form $m!\frac{t^{\deg(A)}A(\v x)}{f(\v x)^{m+1}}$, where $A$ is divisible by the square of some $X_i$, is equivalent
modulo $d\sO_f$ to a $\Z[t]$-linear combination of an admissible term with denominator $f^{m+1}$ and
a numerator of degree $\deg(A)-2$ and admissible terms with denominator $f^m$. 

Thus $m!\frac{t^{\deg(A)}A(\v x)}{f(\v x)^{m+1}}$ is equivalent modulo $d\sO_f$ to a
linear combination of terms with admissible  numerators which have degree at most $1$ in each
variable and denominator $f(\v x)^{k+1}$ with $0 \le k \le m$.
So far we have not used the symmetry of $A$ in the $X_i$, but for the next step this will be important.
Since $\#\sG=n! 2^n$ is invertible in our ring, we can apply symmetrization with respect
to the action of $\sG$ and conclude that $m!\frac{t^{\deg(A)}A(\v x)}{f(\v x)^{m+1}}$
is equivalent to a linear combination of terms $k! t^r s_r(\v X)/f(\v x)^{k+1}$ where 
$0 \le r \le k \le m$ and $s_r(\v X)$ is the elementary symmetric 
polynomial in $X_1,\ldots,X_n$ of degree $r$. When $r=0$ we define $s_0=1$. In the following
computation we use the identities $\sum_{i=1}^n\frac{\partial}{\partial X_i}s_r=(n+1-r)s_{r-1}$,
$\sum_{i=1}^nX_i\frac{\partial}{\partial X_i}s_r=rs_r$
and $\sum_{i=1}^nX_i^2\frac{\partial}{\partial X_i}s_r=s_1s_r-(r+1)s_{r+1}$. These hold
for $r=0,1,\ldots,n$ if we define $s_{-1}=s_{n+1}=0$.
We make the following computation.
\begin{eqnarray*}
\sum_i\theta_i\left(\frac{Y_i\frac{\partial}{\partial X_i}s_r}{f^m}\right)&=&
\sum_i\frac{X_i\frac{\partial}{\partial X_i}s_r}{f^m}+mt\frac{(X_i^2-4)\frac{\partial}{\partial X_i}s_r}
{f^{m+1}}\\
&=&\frac{rs_r}{f^m}+mt\frac{s_1s_r-(r+1)s_{r+1}}{f^{m+1}}-4mt\frac{(n+1-r)s_{r-1}}{f^{m+1}}\\
&=&(r-m)\frac{s_r}{f^m}+m\frac{s_r}{f^{m+1}}-mt(r+1)\frac{s_{r+1}}{f^{m+1}}-
4mt(n+1-r)\frac{s_{r-1}}{f^{m+1}}.
\end{eqnarray*}
In the first equality we used that $\frac{\partial^2}{\partial X_i^2}s_r=0$ and $Y_i^2=X_i^2-4$.
In the third equality we used $ts_1=1-f$. 

After multiplication by $r!(m-1)!t^r$ we get
\be\label{symmetric-reduction}
\bal
m!&\frac{(r+1)!t^{r+1}s_{r+1}}{f^{m+1}}\\
&\is (r-m)(m-1)!\frac{r! t^rs_r}{f^m}+m!
\frac{r!t^rs_r}{f^{m+1}}-
4m!t^2(n+1-r)r\frac{(r-1)!t^{r-1}s_{r-1}}{f^{m+1}}\mod{d\sO_f}.
\eal\ee
Thus we see that modulo $d\sO_f$ every term $m!(r+1)!t^{r+1}s_{r+1}/f^{m+1}$
with $0\le r<n$ equals a $\Z[t]$-linear combination of similar
terms whose numerator or denominator have lower degree. By induction we can continue this reduction 
procedure until we end with a linear combination of terms $k!/f^{k+1}$ with $0\le k\le m$.

To finalize the proof we show that $\theta^n(1/f)$ is a linear combination of $\theta^i(1/f)$,
$i=0,\ldots,n-1$.

To that end we reduce
$(n!)^2t^ns_n/f^{n+1}$ in two ways as linear combination of $r!/f^{r+1}$ with $r=0,\ldots,n$ and show
that they are different. 

First set $r=m=n$ in \eqref{symmetric-reduction}.
Using our definition $s_{n+1}=0$ we find that 
\[
n!\frac{n!t^ns_n}{f^{n+1}}\is 4n! n t^2\frac{(n-1)!t^{n-1}s_{n-1}}{f^n}\mod{d\sO_f}.
\]
Using \eqref{symmetric-reduction}
we see that the right hand side reduces to $t^2$ times a $\Z[t]$-linear combination of
terms $r!/f^{r+1}$ with $r=0,\ldots,n-1$ modulo $d\sO_f$. For the second reduction use the left hand
side and reduce again with the aid of~\eqref{symmetric-reduction}.
This time we get a result that is not a multiple of $t^2$
and therefore not the same as the first reduction. This shows the existence of a non-trivial relation
between $\theta^i(1/f)$, $i=0,1,\ldots,n$. We make this relation more explicit.

Note that the last term on the right hand side of \eqref{symmetric-reduction} reduces to
$t^2$ times a $\Z[t]$-linear combination of $k!/f^{k+1}$ for $k=0,\ldots,m$. We now consider
\eqref{symmetric-reduction} modulo $t^2$,
\[
m!\frac{(r+1)! t^{r+1}s_{r+1}}{f^{m+1}}\is (r-m)(m-1)!
\frac{r!t^rs_r}{f^m}+m!\frac{r!t^rs_r}{f^{m+1}}
\mod{t^2,d\sO_f}.
\]
We assert that
\[
(m-1)!\frac{r! t^rs_r}{f^m}\is 
\theta^r(\theta+1)_{m-r-1}\left(\frac{1}{f}\right)\mod{t^2,d\sO_f}
\]
for all $r<m$. An easy verification shows that this formula is a solution of our recursion
modulo $(t^2,d\sO_f)$ and it also satisfies the starting condition $(m-1)!/f^m=(\theta+1)_{m-1}(1/f)$.

We now conclude that 
\[
n!\frac{n!t^{n}s_n}{f^{n+1}}\is \theta^n\left(\frac{1}{f}\right)\mod{t^2,d\sO_f}.
\]
This reduction is distinct from our first reduction because it is not zero modulo $t^2$. 
Furthermore, we know that $n!\frac{n!t^{n}s_n}{f^{n+1}}$ is a $\Z[t]$-linear combination
of terms $r!/f^{r+1}$ with $r=0,\ldots,n$. By the relation $r!/f^{r+1}=(\theta+1)\ldots(\theta+r)(1/f)$
it is also a $\Z[t]$-linear combination of $\theta^r(1/f)$ for $r=0,\ldots,n$. Using the above
congruence modulo $t^2,d\sO_f$ we see that the coefficient of $\theta^n(1/f)$ equals
$1\mod{t^2}$. Thus we find that there exist $a_i(t)\in\Z[t]$ such that
\[
\sum_{i=0}^n a_i(t) \theta^{n-i}(1/f) \in d\sO_f,
\]
with $a_0(0)=1$ and $a_i(0)=0$ for all $i>0$. We take $D_f(t)=n!a_0(t)$ (to account for inversion of $n!$)
and our proposition is proven.
\end{proof}

\section{Vanishing of $\alpha_1$}\label{sec:vanishing}

In Sections~\ref{sec:HWC}-\ref{sec:hyperoctahedral} we constructed the $p$-adic Frobenius structure for the Picard--Fuchs differential equations $L(y)=0$ of order $n<p$ corresponding to the simplicial and hyperoctahedral families in $n$ dimensions. By Proposition~\ref{cartier2frobenius} the respective Frobenius constants are given by $\alpha_i=p^{-i}\lambda_i(0)$, where $\lambda_i \in \Z_p\lb t \rb$ are the coefficients in the formula
\be\label{cartier-on-f-once-again}
\cartier(1/f) = \sum_{i=0}^{n-1} \lambda_i(t)(\theta^i(1/f))^\sigma \mod {d \hat\sO_{f^\sigma}}.
\ee

Let $n \ge 2$. The fact that $\alpha_1=0$ is crucial for Corollary~\ref{quintic-diagonal}. In the proof of this corollary in the introductory section we said that vanishing follows from~\cite[Proposition 7.7]{BeVlIII}. In this concluding Section we would like to explain how to apply this result to our situation. We work over the ring $\Z_p\lb t \rb$.

\begin{lemma}\label{mod-fil-2-lemma} For the simplicial and hyperoctahedral families in dimension $n \ge 2$ we have 
\begin{itemize} \item[(i)] the quotient $(\hat\sO_f^\circ)^\sG/\fil_2$ is the free $\Z_p\lb t \rb$-module with basis $1/f$ and $\theta(1/f)$;
\item[(ii)] $d \hat\sO_f \cap (\hat\sO_f^\circ)^\sG \subset \fil_2$.
\end{itemize}
\end{lemma}

\begin{proof}
Recall that in the cases under consideration the $n$th Hasse--Witt condition holds due to Theorem~\ref{HWC}. By the results in Section~\ref{sec:simplicial} and~\ref{sec:hyperoctahedral} and Lemma~\ref{independence-mod-hat-Of} we have that $(\hat\sO_f^\circ)^\sG/d\hat\sO_f$ is a free $\Z_p\lb t \rb$-module with the basis $\theta^i(1/f)$, $0 \le i \le n-1$. Hence~(ii) would follow from~(i). We now prove (i). The $n$th Hasse-Witt condition implies in particular the second Hasse-Witt condition. Therefore we have the direct sum decomposition $\hat\sO_f^\circ \cong \sO_f^\circ(2) \oplus \fil_2$. Note that $\sO_f^\circ(2)^\sG$ is spanned by $1/f$ and $\theta(1/f)$ because $\sG$ acts transitively on the set of vertices of the Newton polytope $\Delta$ and the only integral points in $2\Delta^\circ$ are $\v 0$ and the vertices. Using the above direct sum decomposition and doing symmetrization with respect to $\sG$ we obtain (i). 
\end{proof}
 
In~\cite[\S 7]{BeVlIII} we considered the series $A,B,\widetilde\lambda_0,\widetilde\lambda_1 \in \Z_p\lb t \rb$ such that
\be\label{mod-fil-2}\bal
\theta^2(1/f) &= A(t)/f + B(t) \theta(1/f) \mod {\fil_2},\\
\cartier(1/f) &= \widetilde\lambda_0(t)/f + \widetilde\lambda_1(t) (\theta (1/f))^\sigma \mod {\fil_2^\sigma}.
\eal\ee
In loc.cit. we denote them by $\lambda_0$ and $\lambda_1$, but here we need to change the notation to avoid collision with the coefficients in~\eqref{cartier-on-f-once-again}. Though we don't use this fact, let us remark that $\theta^2-B \theta - A$ is a right divisor of our Picard-Fuchs operator $L$ in $\Z_p\lb t \rb[\theta]$ as it follows from (ii) above. Here
\[
A(0)=B(0)=0, \quad \widetilde\lambda_0(0)=1 \text{ and } \widetilde\lambda_1(0)=0
\]
by Propositions~7.6 and~7.7 in loc.cit. respectively. Their proofs are based on congruences for  expansion coefficients of the functions in the identities~\eqref{mod-fil-2}. The desired vanishing result follows from the next lemma.

\begin{lemma}\label{vanishing-lemma} In the above notations we have $\lambda_1(0)=\widetilde \lambda_1(0)$.
\end{lemma}
\begin{proof} Using induction For each $i \ge 2$ we construct the series $A_i,B_i \in \Z_p\lb t \rb$ such that $\theta^i (1/f) = A_i(t)/f + B_i(t) \theta(1/f)\mod {\fil_2}$. Indeed, for $i=2$ we take $A_2=A, B_2=B$ and apply $\theta$ consecutively to obtain the coefficients for higher $i$. It is not hard to show by induction that $A_i(0)=B_i(0)=0$ for each $i$. Due to Lemma~\ref{mod-fil-2-lemma}(ii) identity~\eqref{cartier-on-f-once-again}  also holds modulo $\fil_2^\sigma$. For every $i \ge 2$ we substitute $(\theta^i (1/f))^\sigma$ with $A_i(t^\sigma)/f^\sigma+B_i(t^\sigma)(\theta(1/f))^\sigma$ and conclude that $\widetilde\lambda_1 = \lambda_1 + \sum_{i = 2}^n B_i(t^\sigma)\lambda_i$. Our claim follows from this because $B_i(0)=0$ for all $i$.
\end{proof}


\begin{thebibliography}{xx}
\bibitem{AESZ10} G. Almkvist, C. van Enckevort, D. van Straten, W. Zudilin, 
\emph{Tables of Calabi-Yau equations}, arXiv:math/0507430.
\bibitem{BeVl20I} F. Beukers, M. Vlasenko, \emph{Dwork crystals I}, Int. Math. Res. Notices,
2021 (2021), 8807--8844 
\bibitem{BeVl20II} F. Beukers, M. Vlasenko, \emph{Dwork crystals II}, Int. Math. Res. Notices,
2021 (2021), 4427--4444 
\bibitem{BeVlIII} F. Beukers, M. Vlasenko, \emph{Dwork crystals III}, arXiv:math/2105.14841(2021).
\bibitem{BH89} F.Beukers, G.Heckman, \emph{Monodromy for the hypergeometric function 
${}_nF_{n-1}$\/}, Invent.~Math.~95 (1989), 325--354.
\bibitem{BP01} J. Bryan, R. Pandharipande, \emph{BPS states of curves in Calabi-Yau threefolds},
Geometry \& Toplogy 5 (2001), 287--318. Corrections published in 2002.
\bibitem{COGP91}  P. Candelas, X. de la Ossa, P. Green, L. Parkes,
\emph{An exactly soluble superconformal theory from a
mirror pair of Calabi-Yau manifolds}, Phys. Lett. B 258 (1991), no. 1--2, 118--126.
\bibitem{ERR16} E. Delaygue, T. Rivoal, J. Roques,
\emph{On Dwork's p-adic formal congruences theorem and hypergeometric mirror maps},
Memoirs of the American Mathematical Society, 1163, iii-vi, 94 p. (2016).
\bibitem{Giv95} A. Givental, \emph{Homological geometry and mirror symmetry},
Proceedings ICM-1994 (Z\"urich), vol.1, (1995), 472--480.
\bibitem{IP18} E.N.~Ionel and T.H.~Parker, \emph{The Gopakumar-Vafa formula for symplectic manifolds},
Annals of Math. 187 (2018), 1--64.
\bibitem{KP08} A.Klemm, R.Pandharipande, \emph{Enumerative geometry of Calabi-Yau 4-folds},
Comm. Math. Physics 281 (2008), 621--653.
\bibitem{KSV06}  M.Kontsevich, A.Schwarz, V.Vologodsky,
\emph{Integrality of instanton numbers and p-adic B-model}, Physics Letters B, 637 (2006),
97--101.
\bibitem{KR10} C. Krattenthaler, T.Rivoal,
\emph{On the integrality of Taylor coefficients of mirror maps},
Duke Math. J. 151 (2010), 175--218.
\bibitem{LY95} B. H. Lian and S.-T. Yau, \emph{Mirror maps, modular relations and 
hypergeometric series I}, appeared as \emph{Integrality of certain exponential series},
in: Lectures in Algebra and Geometry, Proceedings of
the International Conference on Algebra and Geometry, Taipei, 1995, M.-C. Kang (ed.), 
Int. Press, Cambridge, MA, 1998, 215--227.
\bibitem{Ro14} J. Roques, \emph{On generalized hypergeometric equations and mirror maps},
Proc. Amer. Math. Soc. 142 (2014), 3153--3167.
\bibitem{Sti03} J. Stienstra, \emph{Ordinary Calabi-Yau-3 Crystals},
Proc. Workshop on Calabi-Yau Varieties and Mirror Symmetry,
Fields Institute Communications volume 38 (2003), 255--271, preprint at arXiv:math/0212061.
\bibitem{DS18} D. van Straten, \emph{Calabi-Yau operators},
in Uniformization, Riemann-Hilbert Correspondence, Calabi-Yau Manifolds \& Picard-Fuchs Equations (Eds. Lizhen Ji, Shing-Tung Yau), Advanced Lectures in Mathematics, Vol. 42, 2018, preprint at arXiv:1704.00164.
\bibitem{Vol07} V.Vologodsky, \emph{Integrality of instanton numbers}, arXiv:0707.4617
\end{thebibliography}
\end{document}